\documentclass[11pt]{article}
\usepackage{latexsym,amsfonts,amsthm,amsmath,amscd,amssymb,color,url}
\usepackage{graphicx}

\newtheorem{lemma}{Lemma}[section]
\newtheorem{thm}[lemma]{Theorem}
\newtheorem{rem}[lemma]{Remark}
\newtheorem{prop}[lemma]{Proposition}
\newtheorem{cor}[lemma]{Corollary}

\newtheorem{example}[lemma]{Example}
\newtheorem{defn}[lemma]{Definition}

\newcommand\matR{{\mathbb{R}}}

\newcommand\matN{{\mathbb{N}}}

\renewcommand{\hbar}{{\overline{h}}}

\newfont{\Got}{eufm10 scaled 1200}

\newcommand{\compo}{\,{\scriptstyle\circ}\,}
\newcommand{\mycap} [1] {\caption{\footnotesize{#1}}}

\newcommand\calG{{\mathcal G}}
\newcommand\calM{{\mathcal M}}

\newcommand\calD{{\mathcal D}}
\newcommand\calO{{\mathcal O}}

\newcommand{\faifig}[3]{\centerline{#2}\mycap{#3}}

\begin{document}

\title{Properly immersed curves in arbitrary surfaces via apparent contours on spines of traversing flows}

\author{Carlo~\textsc{Petronio}\thanks{Partially supported by INdAM through GNSAGA, by
MUR through the PRIN project
n.~2017JZ2SW5$\_$005 ``Real and Complex Manifolds: Topology, Geometry and Holomorphic Dynamics''
and by UniPI through the PRA$\_2018\_22$ ``Geometria e Topologia delle Variet\`a''}}

\maketitle

\begin{abstract}\noindent
Let $\Sigma$ be a compact surface with boundary and $F$ be the set of the orbits of a traversing
flow on $\Sigma$. If the flow is generic, its orbit space is a \emph{spine} $G$ of $\Sigma$,
namely $G$ is a graph embedded in $\Sigma$ and $\Sigma$ is a regular neighbourhood of $G$. Moreover an
extra structure on $G$ turns it into a \emph{flow-spine}, from which one can reconstruct $\Sigma$ and $F$.
In this paper we study properly immersed curves $C$ in $\Sigma$.
We do this by considering generic $C$'s and
their \emph{apparent contour} relative to $F$, namely the set of points of $G$ corresponding to orbits that either are tangent to $C$, or
go through a self-intersection of $C$, or meet the boundary of $C$.
We translate this apparent contour into a decoration of $G$
that allows one to reconstruct $C$, and then we allow $C$ to vary up to homotopy within a fixed generic $F$,
and next also $F$ to vary up to homotopy,
and we identify a finite set of local moves on decorated graphs that translate these homotopies.

\smallskip

\end{abstract}

\newcommand{\Flows}{\textsf{Flows}}
\newcommand{\Genflows}{\textsf{Flows}^0}
\newcommand{\Curves}{\textsf{Curves}(\Sigma)}
\newcommand{\Gencurves}{\textsf{Curves}(\Sigma,F)^0}
\newcommand{\Pairs}{\textsf{Pairs}}
\newcommand{\Genpairs}{\textsf{Pairs}^0}

The notion of apparent contour is most popular in three dimensions, where one considers a closed surface $S$ embedded in Euclidean space, its orthogonal projection
on the horizontal plane, the set of points on the surface at which this projection fails to be have injective differential, and the image $A$ of these points on the plane.
Generically, $A$ is an immersed planar curve with some cusps, and one tries to reconstruct from $A$ and some additional data the projected surface $S$, see~\cite{Gio-contours}
and the references therein. Note that the horizontal plane can be viewed as the orbit space of the vertical unit vector field, and $A$ is then interpreted as the set
of orbits tangent to $S$ at some point. Having this interpretation in mind, in this paper we play the same game in one dimension less,
which of course makes the situation more elementary, but we improve the generality under many respects.  Note that the two-dimensional direct analogue of the above
would be to consider a closed curve embedded in Euclidean plane, a line as the orbit space of a constant vector field, and the apparent contour of
the curve on the line (generically, a finite collection of points). On the contrary, we do the following:
\begin{itemize}
  \item We take as an ambient an \emph{arbitrary compact surface} $\Sigma$ with boundary;
  \item We consider curves in $\Sigma$ that are \emph{immersed} rather than embedded, and we allow them to have a boundary, taking \emph{proper} immersions;
  \item We take on $\Sigma$, instead of a constant planar vector field, an \emph{arbitrary traversing flow}, namely a field
  whose orbits start and end on $\partial\Sigma$;
  \item We note that the orbit space of the flow, under a natural genericity condition, 
  is a graph $G$ more complicated than a line,
  and we give it a decoration that allows one to reconstruct both $\Sigma$ and the flow;
  \item We trace the apparent contour on $G$ of a generic properly immersed curve $C$, and we endow it with a decoration enough to reconstruct $C$;
  \item We consider generic properly immersed curves up to \emph{homotopy through properly immersed curves}, and then  
  generic traversing flows up \emph{homotopy through traversing flows}, 
  and we show how to translate
  these equivalence relations into elementary modifications of the decorated graphs, given by finitely many local combinatorial moves.
\end{itemize}

\bigskip

The study of immersed curves in surfaces is a very classical topic to which many articles were devoted, starting perhaps with the old~\cite{Whitney}, 
and also in recent very years it was the object of considerable research.
However in most cases the ambient surface is the plane or the sphere, 
see~\cite{Valette, ArnoldBook, ArnoldAicardi, Chmutov, CisEtAl, CisEtAl, Francis, GoUme, ItoKodai, Nowik, Polyak, SakaTani, Vassiliev1, Vassiliev2},
and the accent is often on further topological or geometric structures with which the curves interact
(for instance, contact and symplectic geometry in~\cite{ArnoldBook, ArnoldAicardi, ItoKodai}, spacial links in~\cite{Chmutov, CisEtAl}, discriminants 
in~\cite{Vassiliev1, Vassiliev2}, graphs in~\cite{SakaTani}). Some research has been devoted to immersed curves in more general surfaces,
but the setting is always very different from ours. In fact, only closed and orientable surfaces were considered, so the idea of taking proper
immersions was not treated, and very different approaches were taken. For instance, \cite{BartholomewEtAl, Carter}
contain the idea of stabilizing the surface, and \cite{ChangErickson} (which is actually mostly centered on the planar case) permits
simplification through non-immersed curves. The papers~\cite{Carter, Turaev} contain combinatorial encodings of immersed curves in closed 
surfaces, but totally unrelated to those we obtain here (see also~\cite{ItoJKTR}, that builds on~\cite{Turaev} to construct invariants, and actually
mostly deals with planar curves). So, to the best of our knowledge, the approach we take in this work was never explored before.

\bigskip

The original results of this paper
are three main statements, which
are too long to be completely reproduced here,
so we only anticipate them omitting several details.
Each of them shows that a certain collection of topological objects
can be described in combinatorial terms as a certain set of decorated graphs up to
decorated homeomorphism and some local moves. Of course, the interest of the construction
lies not only in the combinatorial presentation itself, but mainly on the geometric construction leading to it,
based on the ideas of taking spines to encode flows and apparent contours on spines to encode curves.

\paragraph{Topological objects}
We introduce here the three collections of topological objects for which our statements provide a combinatorial presentation.

By \emph{surface} we will always mean a compact and connected smooth $2$-manifold $\Sigma$ with non-empty boundary $\partial\Sigma$.
We call \emph{traversing flow} on such a $\Sigma$ the collection $F$ of the orbits of a nowhere-zero smooth vector field $v$ on $\Sigma$,
provided each such orbit is either a point or an interval, and it is oriented according to $v$ in the latter case.
We then define $\Flows$ as the set of all traversing flows $F$ on a (varying) surface $\Sigma$, up to the equivalence relation generated by
diffeomorphisms of surfaces and (for fixed $\Sigma$) the variation of $F$ through traversing flows given by a homotopy of the defining field.

Now let $\Sigma$ be a fixed surface.
We call \emph{curve} a compact smooth $1$-dimensional manifold $c$ with (possibly empty) boundary (a collection of circles and segments),
and \emph{properly immersed curve} in $\Sigma$ the image $C$ of a smooth map
$j:c\to\Sigma$ such that the differential of $j$ is injective everywhere, $j^{-1}(\partial\Sigma)=\partial c$ and
$C$ is transverse to $\partial\Sigma$.
We then define $\Curves$ as the set of all properly immersed curves in $\Sigma$, up to the variation of $C$ generated by a
homotopy of the map $j$ through proper immersions.

Finally, we define $\Pairs$ as the set of all pairs $(F,C)$ where $F$ is a traversing flow and $C$ is a properly immersed curve in the same (varying) surface $\Sigma$,
up to the equivalence relation generated by diffeomorphisms of surfaces and (for fixed $\Sigma$) a simultaneous homotopic variation of $F$ and $C$.

As already known~\cite{KatzFlows}, 
there is a very natural notion of \emph{genericity} for a traversing flow on $\Sigma$. Moreover,
a generic traversing flow $F$ exists on every $\Sigma$, 
and there is a very natural definition of \emph{genericity}
with respect to $F$ of a properly immersed curve $C$ in $\Sigma$. We then define:
\begin{itemize}
\item $\Genflows$ as the set of all generic traversing flows $F$ on a varying $\Sigma$, up to
diffeomorphisms;
\item $\Gencurves$ as the set of all properly immersed curves $C$ in a fixed $\Sigma$ that are generic with respect to a fixed generic $F$,
up to diffeomorphisms of $\Sigma$ mapping $F$ to $F$;
\item $\Genpairs$ as the set of all pairs $(F,C)$ as above in the same variable surface $\Sigma$,
where $F$ is generic and $C$ is generic with respect to $F$, up to diffeomorphisms.
\end{itemize}

For convenience we summarize our notation in the following table:

\bigskip

\begin{tabular}{c|c|c}
\textbf{Set} & \textbf{Relation} & \textbf{Symbol} \\ \hline\hline
$\begin{array}{c}\hbox{Traversing flows}\\ \hbox{on surfaces}\end{array}$  &
        Homotopy    &
                $\Flows$ \\ \hline
$\begin{array}{c}\hbox{Generic traversing}\\ \hbox{flows on surfaces}\end{array}$   &
        Diffeomorphism    &
                $\Genflows$ \\ \hline\hline
$\begin{array}{c}\hbox{Properly immersed curves}\\ \hbox{in a surface }\Sigma\end{array}$ &
        Homotopy &
                $\Curves$ \\ \hline
$\begin{array}{c}\hbox{Properly immersed curves}\\ \hbox{in }\Sigma\hbox{ generic for a flow }F\end{array}$ &
        Diffeomorphism &
                $\Gencurves$
                \\ \hline\hline
$\begin{array}{c}\hbox{Pairs (traversing flow},\\ \hbox{properly immersed curve)}\\ \hbox{on the same surface}\end{array}$ &
        $\begin{array}{c}\hbox{Simultaneous}\\ \hbox{homotopy}\end{array}$ &
                $\Pairs$ \\ \hline
$\begin{array}{c}\hbox{Generic pairs (traversing flow},\\ \hbox{properly immersed curve)}\\ \hbox{on the same surface}\end{array}$ &
        Diffeomorphism &
                $\Genpairs$ \\
\end{tabular}

\paragraph{Graphs, reconstruction, and moves}
If $X$ is one of the sets $\Flows$, $\Curves$, or $\Pairs$, let $X^0$ be the corresponding $\Genflows$, $\Gencurves$ for some $F$, or $\Genpairs$.
Note that there is an obvious projection $\pi(X^0):X^0\to X$.
We will define below the following objects:
\begin{itemize}
  \item A set $\calG(X^0)$ of finite decorated graphs, up decorated homeomorphism;
  \item A map $\varphi(X^0):\calG(X^0)\to X^0$ based on an explicit construction;
  \item A finite set $\calM(X^0)$ of local combinatorial moves on $\calG(X^0)$.
\end{itemize}
And we will prove (see Theorems~\ref{flows:thm},~\ref{curves:thm}, and~\ref{pairs:thm}):
\begin{thm}\label{summarising:thm}\
\begin{itemize}
\item $\varphi(X^0):\calG(X^0)\to X^0$ is bijective;
\item $\pi(X^0): X^0\to X$ is surjective;
\item Two graphs in $\calG(X^0)$ have the same image in $X$ under the composition $\pi(X^0)\compo\varphi(X^0)$
if and only if they are related by a finite combination of moves in $\calM(X^0)$.
\end{itemize}
\end{thm}

\begin{cor}\label{presentation:cor}
$X$ can be identified to the quotient of $\calG(X^0)$ under the equivalence relation generated by $\calM(X^0)$.
\end{cor}

\begin{rem}\label{closed:curves:rem}
\emph{It is easy to identify the subsets of $\calG(\Gencurves)$ consisting of the graphs that give
\emph{closed} curves, or \emph{embedded} curves, or \emph{closed embedded} curves,
and the moves in $\calM(\Gencurves)$ that apply to these subsets. Hence 
from Theorem~\ref{summarising:thm}, by restriction, we get 
combinatorial presentations of the sets of all closed immersed curves in $\Sigma$, 
of all embedded curves in $\Sigma$, and of all closed embedded curves in $\Sigma$, always up to homotopy.
And similarly taking subsets of $\calG(\Genpairs)$ and $\calM(\Genpairs)$
we get similar presentations with $\Sigma$ allowed to vary up to diffeomorphism and $F$ up to homotopy.
See Remarks~\ref{graphs:for:embedded:closed:curves:rem},~\ref{moves:for:embedded:closed:curves:rem}, and~\ref{embedded:closed:pairs:rem}.}
\end{rem}

Note that the combinatorial presentation of Corollary~\ref{presentation:cor}, or its restrictions just dicussed,
provide a potential framework for the construction of invariants of properly immersed curves (or closed ones, or embedded ones,
or closed embedded ones).

\bigskip

We now very informally allude to the constructive nature of the maps $\varphi(X^0)$.
For $F\in\Genflows$ on $\Sigma$, the graph $G$ that gives $F$ via $\varphi(\Genflows)$ 
is the orbit space of $F$, endowed
with a certain decoration (this was already known to Katz~\cite{KatzFlows}, and the decoration includes
a structure of train-track, see~\cite{FLP}).
Moreover $G$ is a spine of $\Sigma$, namely it has a natural embedding in $\Sigma$,
and a regular neighbourhood of $G$ in $\Sigma$ is diffeomorphic to $\Sigma$.
For this reason $G$ is called a \emph{flow-spine}.

For $C\in\Gencurves$, the graph that gives $C$ via $\varphi(\Gencurves)$  
is obtained by adding bivalent decorated vertices to the above graph $G$. These new vertices
encode the apparent contour of $C$ with respect to $F$, the double points of $C$, and its points on $\partial\Sigma$.
The map $\varphi(\Genpairs)$ is just given by the collection of all $\varphi(\Gencurves)$ as $\Sigma$ and $F$ vary.

\bigskip

We also mention that the moves $\calM(X^0)$ arise by translating in combinatorial terms the first-order violations
of genericity that can occur along a homotopy of a traversing flow, or of a properly immersed curve within a generic flow,
or of a simultaneous homotopy of a traversing flow and a properly immersed curve.

\bigskip

We conclude this introduction saying that
the idea of describing a manifold by means of the orbit space of a generic traversing flow was exploited in three dimensions in~\cite{BP-LNM} and~\cite{Ishii}.
This approach was later developed in various directions by,
among others, S.~Baseilhac, R.~Benedetti, I.~Ishii,  
M.~Ishikawa, Y.~Koda, H.~Naoe, and the author
(a complete account of these results
is not relevant for us here).
We are planning in the future to extend the methods of the present
paper to study surfaces embedded in an arbitrary $3$-dimensional manifold via
their apparent contour on the flow-spine of a traversing flow on the manifold.


\section{Traversing flows on surfaces\\ and their flow-spines}

In this section we introduce traversing flows on surfaces and their combinatorial encoding by decorated uni-trivalent graphs up certain moves.
As already mentioned, part of what we show was already known to Katz.

As stated in the introduction, all our surfaces $\Sigma$ will be smooth, connected, compact, and with boundary $\partial\Sigma\neq\emptyset$.
We call \emph{traversing flow on $\Sigma$} a partition $F$ of $\Sigma$ such that:
\begin{itemize}
\item Each element of $F$ is either a point or an \emph{oriented} arc;
\item There exists a smooth nowhere-zero vector field $v$ on $\Sigma$ such that $F$ is the collection of the oriented orbits of $v$.
\end{itemize}
Note that:
\begin{itemize}
  \item If a point is in $F$ then it is on $\partial\Sigma$;
  \item The ends of an arc in $F$ are on $\partial\Sigma$, but also inner points can be.
\end{itemize}
We now define $\Flows$ as the quotient of the set of all traversing flows $F$ on an arbitrary (variable) surface $\Sigma$
under the equivalence relation $\sim$ given by \emph{diffeomorphism and homotopy
through traversing flows} on $\Sigma$, namely that generated as follows:
\begin{itemize}
  \item $F_0\sim F_1$ if there exists a diffeomorphism $u:\Sigma_0\to\Sigma_1$
  of the underlying surfaces
  such that $u(F_0)=F_1$ respecting the orientation of the arcs;
  \item $F_0\sim F_1$ on the same $\Sigma$ if there exists a family $\{v_t:\ t\in[0,1]\}$ of nowhere-zero vector fields on $\Sigma$ depending smoothly on $t$ such that
for all $t$ the orbits of $v_t$ are a traversing flow $F_t$.
\end{itemize}

We now introduce the notion of genericity for a flow $F$ in $\Sigma$ already alluded to in the introduction. For short,
$F$ is generic if each of its orbits is tangent to $\partial\Sigma$ at one point at most, with order-$2$ contact. This means that
for every $f_0\in F$ there exists an open subset $U$ of $\Sigma$ which contains $f_0$ and
is a union of elements of $F$ such that $(U,f_0,\{f\in F:\ f\subset U\})$ is diffeomorphic to $(\overline{U},\overline{f}_0,\overline{F})$ where
$\overline{U}$ is one of three specific subsets  of $\matR^2$ described below, while $\overline{F}$ is always the
traversing flow on $\overline{U}$ generated by the constant upward vertical vector field,
and $\overline{f}_0$ is always the orbit through $(0,0)$. The models for $\overline{U}$ are the following, and
correspondingly $f_0$ is termed as indicated (see Fig.~\ref{genericflow:fig}):
\begin{figure}
\faifig
{}
{\includegraphics[scale=0.6]{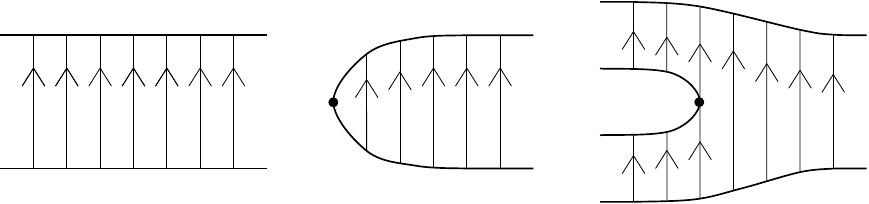}}
{Non-tangent orbits, a convex tangent orbit, and a concave tangent orbit.\label{genericflow:fig}}
\end{figure}

\begin{itemize}
  \item $\overline{U}=(-1,1)\times[-1,1]$; we say $f_0$ is a non-tangent orbit;
  \item $\overline{U}=\{(x,y)\in\matR^2:\ y^2\leqslant x<1\}$; we say $f_0$ is a convex tangent orbit;
  \item $\overline{U}=\{(x,y)\in\matR^2:\ 2|x|<1,\ |y|\leqslant1,\ x\geqslant -y^2\}$; we say $f_0$ is a concave tangent orbit.
\end{itemize}

We now note that the set of all traversing flows on $\Sigma$ is the quotient of a space of smooth vector fields on $\Sigma$
($F$ represents the class of all the vector fields whose orbits are the elements of $F$). The latter space can be given the
$\textrm{C}^{\infty}$ topology, so the former can be given the quotient topology.

\begin{rem}
\emph{Katz provided a formal proof in~\cite{KatzFlows}
of the fact that in the space of all traversing flows the subspace of generic ones is open and dense.
In particular, some arbitrarily small perturbation of any traversing flow is generic, and a sufficiently
small perturbation of a generic traversing flow remains generic. See also~\cite{KatzModels, KatzBook} for an extension of this
result to any dimension, and~\cite{KatzHolography} for interesting related results.}
\end{rem}

We define $\Genflows$ as the set of all generic traversing flows $F$ on an arbitrary (variable) $\Sigma$
up to diffeomorphisms.
To introduce the combinatorial counterpart $\calG(\Genflows)$ of this set, we denote by $\calO(F)$ the orbit space of $F$, with the natural topology.

\begin{lemma}
If $F$ is a generic flow on $\Sigma$ then $\calO(F)$ is a finite connected graph with vertices of valence $1$ or $3$.
\end{lemma}

\begin{proof}
First, note that the points at which $F$ is tangent to $\partial\Sigma$ are isolated, so they are finite in number. Now,
let $\partial_+\Sigma$ be the set of points of $\partial\Sigma$ at which $F$ is entering $\Sigma$ or is tangent to $\partial\Sigma$, see
Fig.~\ref{enteringboundary:fig}.
\begin{figure}
\faifig
{}
{\includegraphics[scale=0.6]{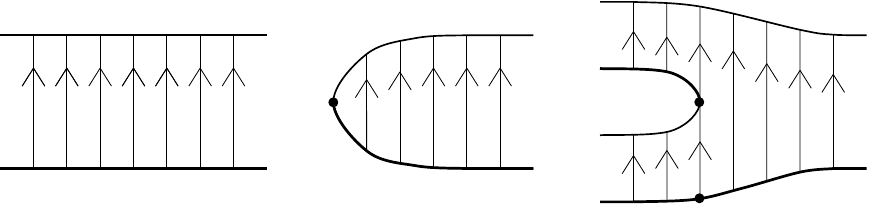}}
{$\partial_+\Sigma$ is the thicker portion of $\partial\Sigma$, and the orbit space
is obtained by merging together the two points marked by dots on the right.
\label{enteringboundary:fig}}
\end{figure}
Then $\partial_+\Sigma$ is a finite union of arcs and every orbit of $F$ meets $\partial_+\Sigma$ at precisely one point, except the
concave orbits, that meet $\partial_+\Sigma$ at one internal point of an arc and at an endpoint of another arc (possibly the same one).
So the orbit space of $F$ is obtained from $\partial_+\Sigma$ by merging together these two points for every concave orbit.
This shows that $\calO(F)$ is a finite graph with vertices of valence $1$ or $3$. A path in $\Sigma$ joining two points of two orbits
of $F$ projects in $\calO(F)$ to a path joining the points representing the orbits, so $\calO(F)$ is connected.
\end{proof}

Note that if $\Sigma$ is $S^1\times[0,1]$ and $F$ is parallel to $[0,1]$ then $\calO(F)$ is a circle without vertices,
that we declare from now on to be a legal graph.

A useful alternative way to realize $\calO(F)$ within $\Sigma$ is shown in Fig.~\ref{flowspine:fig},
\begin{figure}
\faifig
{}
{\includegraphics[scale=0.6]{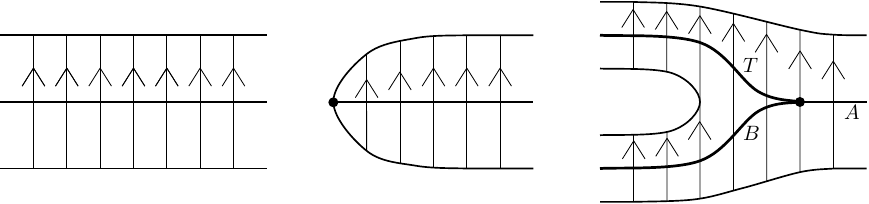}}
{Portions of a flow-spine
\label{flowspine:fig}}
\end{figure}
where we also introduce at the trivalent vertices a singular smooth structure
and a labelling of the germs of edge, where $T/B/A$ stand respectively for Top/Bottom/Alone.
With these extra structures, $\calO(F)$ will be called a \emph{flow-spine} of $F$.
Note that the labels $T/B/A$ already allow to reconstruct the singular smooth structure,
but in all our pictures we will also represent the latter, which is very
intuitive.

\begin{rem}
\emph{The singular smooth structure we introduce at the trivalent vertices of our graphs
gives the $1$-dimensional analogue of the notion of \emph{branched surface} (see~\cite{BP-LNM} and the references therein).
It was given the name of \emph{train-track} and deeply exploited by W.~P.~Thurston (see~\cite{FLP} and the references therein), but
for different aims than ours.}
\end{rem}

We now proceed in the opposite direction, from graphs to flows, defining $\calG(\Genflows)$ and
$\varphi(\Genflows):\calG(\Genflows)\to\Genflows$.
The idea of using the labelling to encode the flow is due to
Katz~\cite{KatzFlows}.

\begin{defn}
\emph{We define $\calG(\Genflows)$ as the set of finite connected graphs with vertices of valence $1$ or $3$, where
the three germs of edge at each vertex of valence $3$ are decorated by the three labels $T,B,A$.
The elements of $\calG(\Genflows)$ are viewed up to decorated homeomorphism.}
\end{defn}

\begin{defn}
\emph{We define
$\varphi(\Genflows):\calG(\Genflows)\to\Genflows$ as follows. If $G\in\calG(\Genflows)$ has no vertex, so it is a circle,
$\varphi(\Genflows)(G)$ is $S^1\times[0,1]$ with flow parallel to $[0,1]$.
Otherwise $\varphi(\Genflows)(G)$ is constructed as follows:
\begin{itemize}
\item Cut each edge of $G$ into halves by means of a small dash, as in Fig.~\ref{edgecut:fig};
\begin{figure}
\faifig
{}
{\includegraphics[scale=0.6]{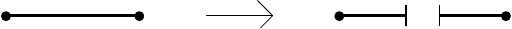}}
{Cutting an edge into halves.\label{edgecut:fig}}
\end{figure}
\item Associate to each resulting portion of graph a portion of surface with boundary and a flow on it as in
Fig.~\ref{flowportions:fig};
\begin{figure}
\faifig
{}
{\includegraphics[scale=0.6]{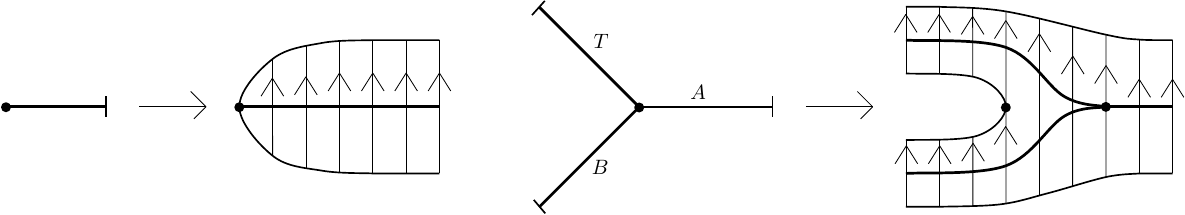}}
{Portions of flows corresponding to portions of graphs.
\label{flowportions:fig}}
\end{figure}
\item Glue these portions of surface with flow along boundary arcs of the flows
so to reconstruct the original graph and to respect the orientation of the glued arcs.
\end{itemize}
The result is of course a generic traversing flow $F$ on a surface $\Sigma$, well-defined up to diffeomorphism
and depending only on $G$ up to decorated homeomorphism.}
\end{defn}

\begin{rem}
\emph{The portion of graph shown on the left in Fig.~\ref{flowportions:fig} is symmetric
with respect to the reflection in a horizontal line, while the corresponding portion of surface with flow is not,
so an ambiguity issue
might seem to arise. But the gluings should match the
orientation of the arcs, so the result of the construction is well defined.}
\end{rem}

\begin{example}
\emph{Figure~\ref{mobius:fig}
\begin{figure}
\faifig{}
{\includegraphics[scale=0.6]{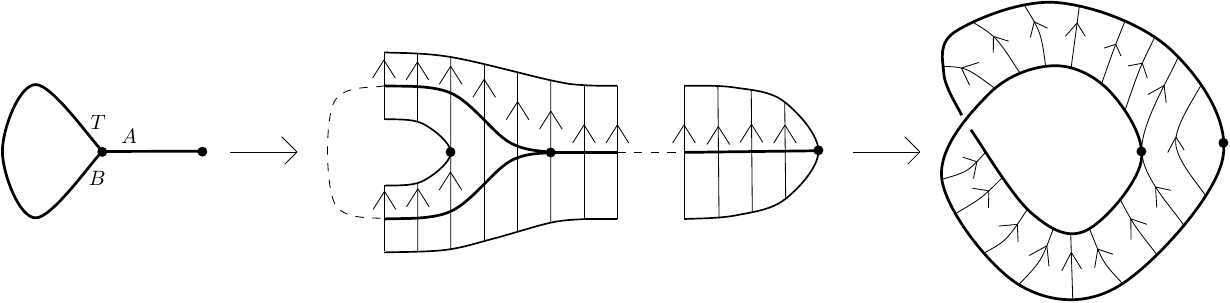}}
{A flow-spine encoding a generic flow on the M\"obius strip. Here and in Fig.~\ref{annulus:fig} a dashed curve indicates the flow lines to be identified (respecting the orientation).
\label{mobius:fig}}
\end{figure}
illustrates how $\varphi(\Genflows)$ works on an easy flow-spine, giving a generic flow on the M\"obius strip.}
\end{example}

\begin{example}
\emph{Figure~\ref{annulus:fig}
\begin{figure}
\faifig{}
{\includegraphics[scale=0.6]{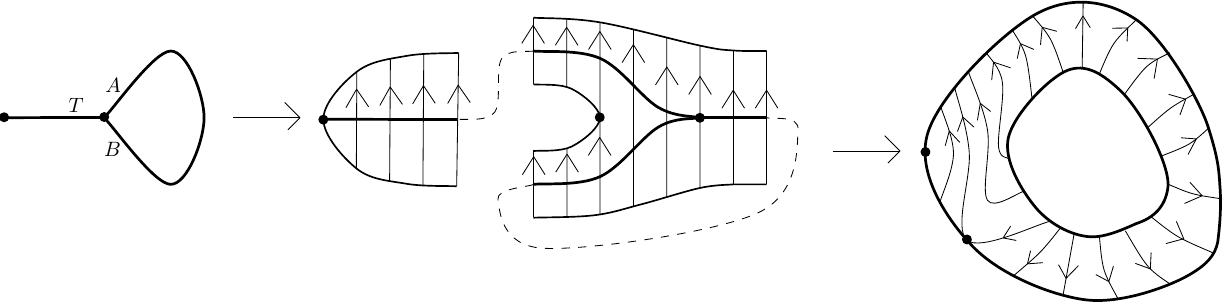}}
{A flow-spine encoding a generic flow on the annulus.
\label{annulus:fig}}
\end{figure}
shows how a mere change of labelling on the flow-spine of Fig.~\ref{mobius:fig} produces a completely different topological result.}
\end{example}

To proceed, we introduce on $\calG(\Genflows)$ the four combinatorial moves described in Fig.~\ref{s-moves:fig}, collectively denoted by $\calM(\Genflows)$.
\begin{figure}
\faifig{}
{\includegraphics[scale=0.6]{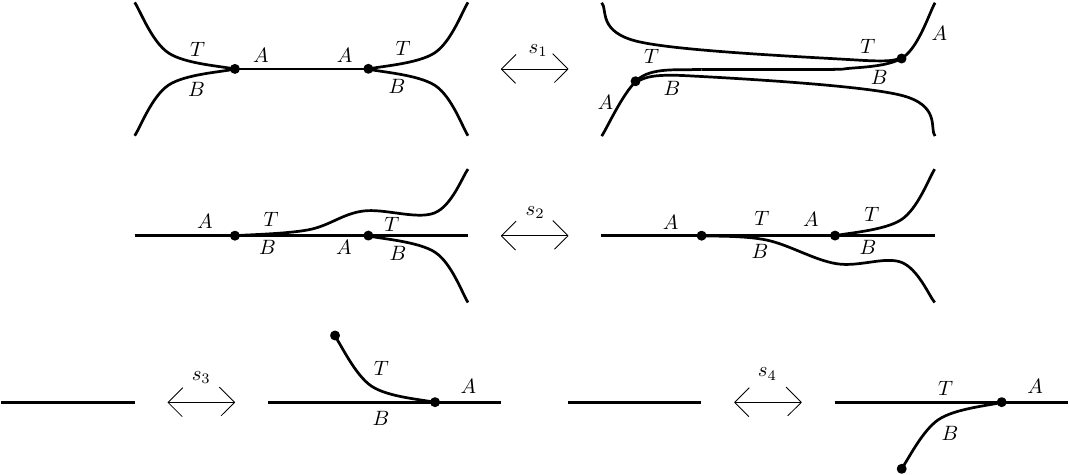}}
{Moves on flow-spines.\label{s-moves:fig}}
\end{figure}
Note that for the moves $s_1$, $s_3$ and $s_4$, the configuration on the left in Fig.~\ref{s-moves:fig} has a vertical symmetry in the plane, wheras that on the right does not, which
implies that we have two ways to apply these moves left to right, and only one right to left.
(For more on this, see Remark~\ref{left-right:rem} below.)

Now let $\pi(\Genflows): \Genflows\to \Flows$ associate
to a generic traversing flow up to diffeomorphism its equivalence class up to homotopy.
The first item of the next statement was already known to Katz~\cite{KatzFlows}.

\begin{thm}\label{flows:thm}\
\begin{itemize}
\item $\varphi(\Genflows):\calG(\Genflows)\to \Genflows$ is bijective;
\item $\pi(\Genflows): \Genflows\to \Flows$ is surjective;
\item Two graphs in $\calG(\Genflows)$ have the same image in $\Flows$ under the composition $\pi(\Genflows)\compo\varphi(\Genflows)$
if and only if they are related by a finite combination of moves in $\calM(\Genflows)$.
\end{itemize}
\end{thm}

\begin{proof}
The inverse of $\varphi(\Genflows)$ is the map that associates to a generic traversing
flow $F$ on a surface $\Sigma$ its decorated orbit space $\calO(F)$, as shown in Fig.~\ref{flowspine:fig}.
Note that if $F$ has no tangency to $\partial\Sigma$ then $\Sigma$ is $S^1\times[0,1]$ with $F$ parallel to $[0,1]$.

Surjectivity of $\pi(\Genflows)$ follows from the fact that any traversing flow can be slightly perturbed to a generic one, and
this perturbation is a homotopy of the generating vector field.

Turning to the third item, we recall that a traversing flow $F$ on $\Sigma$ is generic if each of its orbits is tangent to $\partial\Sigma$
at one point at most, with order-$2$ contact. This implies that along an evolution of $F$ through traversing flows generated by a homotopy
of the defining vector fields, up to small perturbation only the following
elementary catastrophes happen, and at different times:
\begin{enumerate}
  \item One orbit $f$ is tangent to two points of $\partial\Sigma$, with order-$2$ contact;
  \item One orbit $f$ is tangent to one point of  $\partial\Sigma$, with order-$3$ contact.
\end{enumerate}
The first type of catastrophe can actually happen in two essentially distinct ways, depending on whether the two
folds of $\partial\Sigma$ tangent to $f$ lie to opposite sides of $f$ or to the same side, see
the first two portions of Fig.~\ref{nongenericflows:fig}.
\begin{figure}
\faifig{}
{\includegraphics[scale=0.6]{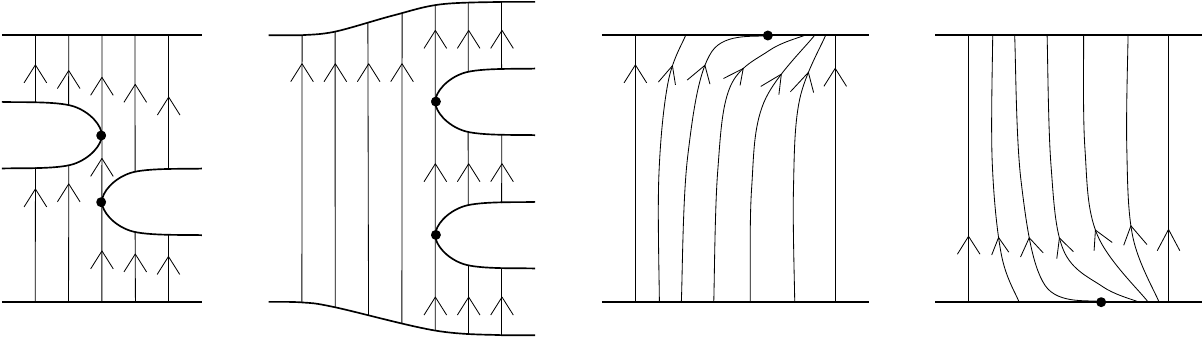}}
{Non-generic flows.\label{nongenericflows:fig}}
\end{figure}
Similarly, the second type can happen in two ways, depending on whether the tangency is at the end or at the start of $f$, see
the last two portions of Fig.~\ref{nongenericflows:fig}.
Now Fig.~\ref{moves1:fig}
\begin{figure}
\faifig{}
{\includegraphics[scale=0.6]{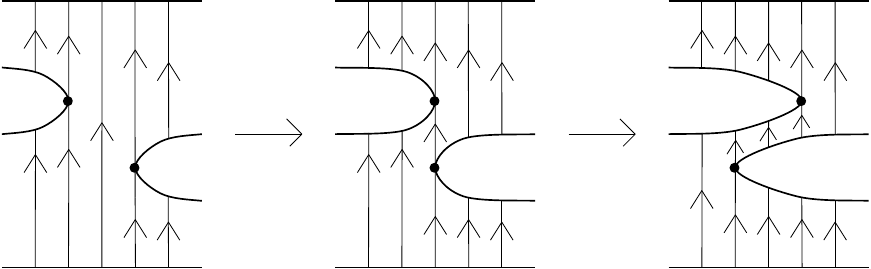}}
{The elementary catastrophe leading to the move $s_1$.\label{moves1:fig}}
\end{figure}
shows how $F$ evolves as it goes through a catastrophe as in the first portion of Fig.~\ref{nongenericflows:fig},
which readily translates at the level of flow-spines into the move $s_1$ of
Fig.~\ref{s-moves:fig}. Similarly, the second portion of Fig.~\ref{nongenericflows:fig} gives the move $s_2$
(see below Fig.~\ref{si-2-move-proof:fig} for an illustration of an enhanced version of this move).
In Fig.~\ref{moves3:fig}-top we show how $F$ evolves as it goes through a catastrophe as in the third portion of Fig.~\ref{nongenericflows:fig}.
This evolution is best visualized by redrawing the picture so that the flow is constant vertical upward, and $\Sigma$ changes shape, as
done in Fig.~\ref{moves3:fig}-bottom, which readily implies that at the level of flow-spines we get the move $s_3$.
\begin{figure}
\faifig{}
{\includegraphics[scale=0.6]{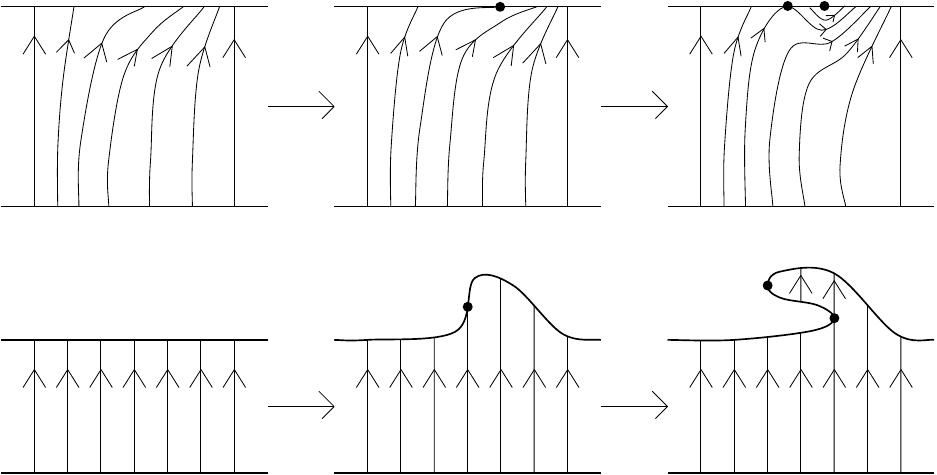}}
{Two ways to visualize the elementary catastrophe leading to the move $s_3$.\label{moves3:fig}}
\end{figure}
Similarly the fourth portion of Fig.~\ref{nongenericflows:fig} gives $s_4$.
The proof is complete.
\end{proof}

\begin{rem}\emph{No univalent vertex of any $G\in\calG(\Genflows)$ joined to a trivalent vertex through the germ of edge labelled $A$ is involved
in any of the moves $\calM(\Genflows)$. But after some move $s_1$ it can perhaps be removed by a move $s_3$ or $s_4$, as
Fig.~\ref{nonstable:fig}
\begin{figure}
\faifig{}
{\includegraphics[scale=0.6]{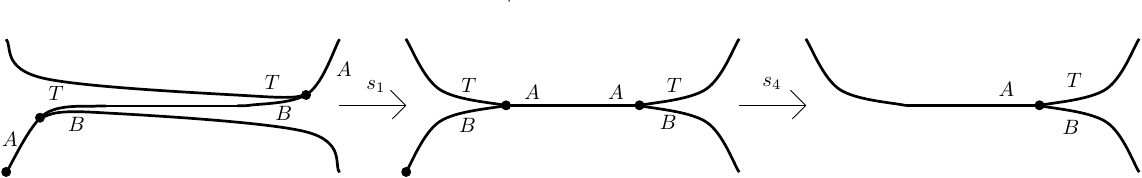}}
{A combination of moves in $\calM(\Genflows)$.
\label{nonstable:fig}}
\end{figure}
shows.}
\end{rem}

We conclude this section with a first application of our combinatorial representation of $\Flows$:

\begin{lemma}
If $G\in\calG(\Genflows)$ has $u$ univalent and $t$ trivalent  vertices, and if $\varphi(\Genflows)$ maps $G$ to a flow $F$ on $\Sigma$,
then $\chi(\Sigma)=\frac12(u-t)$.
\end{lemma}

\begin{proof}
Since $G$ is a deformation retract of $\Sigma$, it suffices to show that $\chi(G)=\frac12(u-t)$.
If $G$ has $e$ edges, counting their ends we get $2e=u+3t$, so
$$\chi(G)=(u+t)-e=(u+t)-\frac12(u+3t)=\frac12(u-t).$$
\end{proof}

\begin{prop}\label{disc:prop}
Let $G\in\calG(\Genflows)$ be mapped by
$\varphi(\Genflows)$ to a flow $F$ on $\Sigma$,
with $\chi(\Sigma)\geqslant 0$. Then a combination of moves in $\calM(\Genflows)$ trasforms $G$ into a $G'$ as follows:
\begin{itemize}
\item If $\Sigma$ is the disc, $G'$ is a segment;
\item If $\Sigma$ is the annulus, $G'$ is one of the graphs in
Fig.~\ref{anngraphs:fig};
\begin{figure}
\faifig{}
{\includegraphics[scale=0.6]{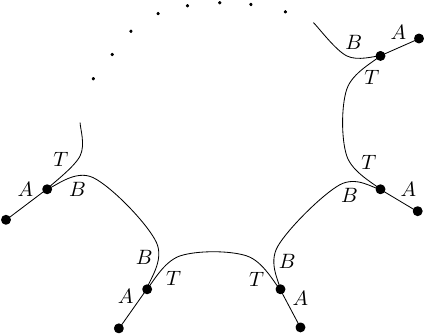}}
{A flow-spines defining a flow on the annulus, with $2k$ univalent and $2k$ trivalent vertices, for $k\in\matN$.\label{anngraphs:fig}}
\end{figure}
\item If $\Sigma$ is the M\"obius strip, $G'$ is the flow-spine of Fig.~\ref{mobius:fig}-left.
\end{itemize}
\end{prop}

\begin{proof}
Suppose that $G$ has $u$ univalent and $t$ trivalent  vertices.

If $\Sigma$ is the disc, $u-t=2$.
If $t=0$ then $G$ is already a segment. If $t>0$ we note that every univalent vertex $x$
is joined to a trivalent one $y$. This junction is realized through one of the three germs of edge at $y$,
and for the two of them with labels $T$ and $B$ we can apply a move $s_3$ or $s_4$ to reduce $u$ and $t$ by 1.
But there are $u$ univalent vertices and $t=u-2$ germs labeled $A$, so a reducing move $s_3$ or $s_4$
applies somewhere, and we conclude by induction.

We treat the cases of the annulus and the M\"obius strip together.  We have $u=t$, and if $u=t=0$ we have that $G$ is a circle,
which is the flow-spine in Fig.~\ref{anngraphs:fig} for $k=0$. Now suppose $u=t=1$. If the univalent vertex is joined to the $T$ or $B$ germ of edge
at the trivalent vertex, we can apply $s_3$ or $s_4$ getting back to $u=t=0$, otherwise we have the flow-spine of Fig.~\ref{mobius:fig}-left.
Next, take $u=t>1$. If a univalent vertex is joined to a trivalent one through a $T$ or $B$ germ,
which happens in particular if two univalent vertices are joined to the same trivalent one,
we can apply $s_3$ or $s_4$ and proceed inductively.
Hence each univalent vertex is joined to a trivalent one, and conversely, and
all the junctions involve the $A$ germs of edge. To conclude, we note that if there
is a junction between a $T$ germ of a trivalent vertex and a $B$ germ of another one then we can apply a move $s_1$ and then
$s_3$ or $s_4$ to reduce $u=t$, as in Fig.~\ref{nonstable:fig}. The only possibilities left are those of
Fig.~\ref{anngraphs:fig}.
\end{proof}

\begin{cor}\label{disc:cor}
If $\chi(\Sigma)\geqslant 0$ then every traversing flow on $\Sigma$ is homotopic through traversing flows to one and only one of the following:
\begin{itemize}
\item If $\Sigma$ is the disc, the flow in Fig.~\ref{easyflows:fig}-left;
\item If $\Sigma$ is the annulus, one of the flows of the sequence of Fig.~\ref{easyflows:fig}-right;
\begin{figure}
\faifig{}
{\includegraphics[scale=0.6]{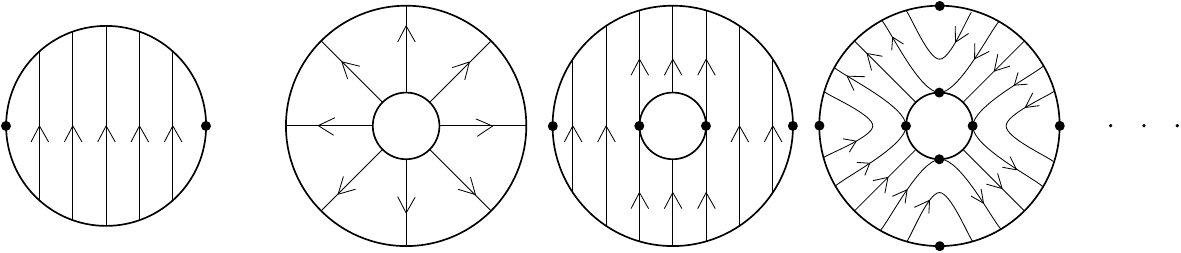}}
{A flow on the disc and a sequence of flows on the annulus.\label{easyflows:fig}}
\end{figure}
\item If $\Sigma$ is the M\"obius strip, the flow in Fig.~\ref{mobius:fig}-right.
\end{itemize}
\end{cor}

\begin{proof}
The flows in the statement are those defined via $\varphi(\Genflows)$ by the flow-spines of Proposition~\ref{disc:prop}.
We then only need to show that when $\Sigma$ is the annulus the flows of the sequence of Fig.~\ref{easyflows:fig}-right
are pairwise inequivalent. This is because if we realize the annulus in $\matR^2$ with $S^1$ as one of the boundary components,
and $v$ is a unitary vector field generating a flow $F$, then $v$ restricts to a map from $S^1$ to itself, which has a well-defined degree.
This degree only depends on $F$ up to homotopy, and its values for the flows in
the sequence of Fig.~\ref{easyflows:fig}-right are $1,\ 0,\ -1,\ \ldots$
\end{proof}


\section{Apparent contours of immersed curves\\ in a fixed flow}

In this section we show how to describe an immersed curve in a fixed surface $\Sigma$ by means of its apparent contour with respect to a given
generic traversing flow. To begin we prove the following:

\begin{prop}
There always exists a generic traversing flow $F$ on $\Sigma$.
\end{prop}

\begin{proof}
Taking a triangulation of $\Sigma$ and collapsing it as long as no triangle is left (recall that $\partial\Sigma\neq\emptyset$) we see that
there exists a graph $X\subset\Sigma$ which is a spine of $\Sigma$, namely $\Sigma$ can be identified to a
regular neighbourhood of $X$ in $\Sigma$ (but the abstract homeomorphism type of $X$ does not determine $\Sigma$, one needs the embedding of $X$ in $\Sigma$).
We can assume that $X$ does not reduce to a point, that it has vertices of valence $1$ and $3$ only,
that every edge of $X$ is a smooth arc, and that the edges are transverse to
each other at the vertices. Then we can construct $F$ on regular neighbourhoods of the vertices by making some arbitarty choice as in Fig.~\ref{arbatverts:fig},
\begin{figure}
\faifig{}
{\includegraphics[scale=0.6]{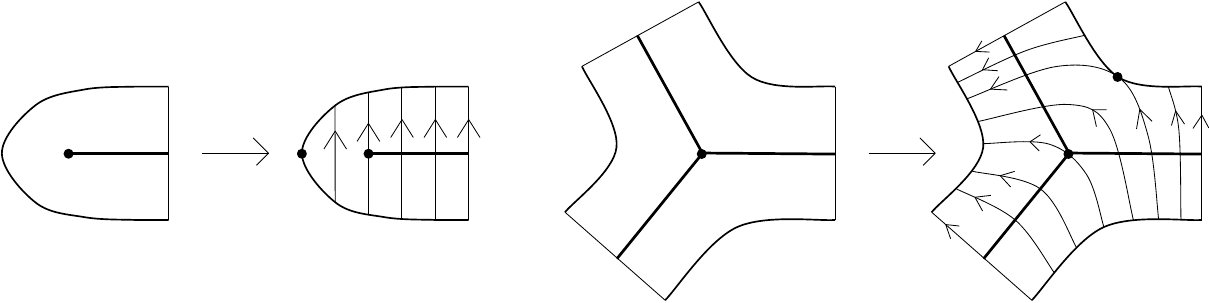}}
{A choice of $F$ near the vertices.\label{arbatverts:fig}}
\end{figure}
and finally extend $F$ across the edges, as in Fig.~\ref{extendacross:fig}.
\begin{figure}
\faifig{}
{\includegraphics[scale=0.6]{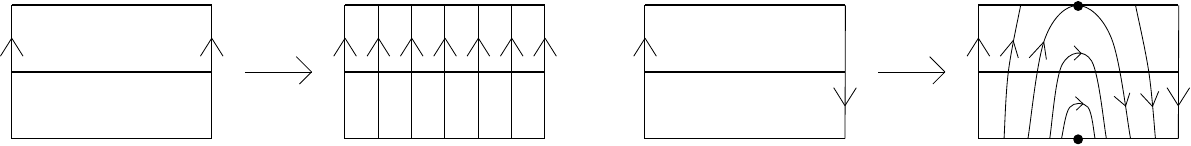}}
{Extending $F$ across the edges.\label{extendacross:fig}}
\end{figure}
\end{proof}

In the rest of the section we fix a generic traversing flow $F$ on $\Sigma$.
Recall that a \emph{properly immersed curve} in $\Sigma$ is the image $C$ of a smooth map
$j:c\to\Sigma$, where $c$ is an abstract curve, $j$ has injective differential, $j^{-1}(\partial\Sigma)=\partial c$, and
$C$ is transverse to $\partial\Sigma$. Moreover
$\Curves$ is the set of all such $C$'s up to the variation generated by a
homotopy of $j$ through proper immersions.

\begin{defn}\label{gencurve:defn}
\emph{A properly immersed curve $C$ in $\Sigma$ is \emph{generic} for $F$ if:
\begin{itemize}
\item[(D)] The only multiple points of $C$ are transverse double ones not on $\partial\Sigma$, so they form a finite set $\calD(C)$;
\item[(X)] The points of $\Sigma$ in $F$ (the point-orbits) do not belong to $C$;
\item Every point $P$ of an arc element $f$ of $F$ satisfies one following:
\begin{itemize}
\item[(0)] $P\not\in C$, and $P\not\in\partial\Sigma$ or $P\in\partial\Sigma$ and $f$ is transverse to $\partial\Sigma$ at $P$;
\item[(1)] $P\in C$, $P\not\in\partial\Sigma$, $P\not\in\calD(C)$ and $C$ is transverse to $f$ at $P$;
\item[(2)] $P\in C$, $P\not\in\partial\Sigma$, $P\not\in\calD(C)$ and $C$ is tangent to $f$ at $P$ with an order-2 contact;
\item[(3)] $P\in C$, $P\not\in\partial\Sigma$, $P\in\calD(C)$ and both the strands of $C$ at $P$ are transverse to $f$;
\item[(4)] $P\in C$, $P\in\partial\Sigma$ and $C$ is transverse to $f$ at $P$;
\item[(5)] $P\not\in C$, $P\in\partial\Sigma$ and $f$ is tangent to $\partial\Sigma$ at $P$ (with order-2 contact, by the genericity of $F$).
\end{itemize}
Moreover, every $f$ contains at most one point of the types (2) to (5).
\end{itemize}}
\end{defn}

Pictures illustrating the point types of a generic curve
will be provided in Fig.~\ref{curveportions:fig}
together with their combinatorial encoding.

Now consider the set of all properly immersed curves in $\Sigma$ that are images of some $j:c\to \Sigma$ for a fixed $c$.
This is the quotient of the set of all such $j$'s, where two are identified if they have the same image.
This space of $j$'s can be given the $\textrm{C}^{\infty}$ topology, so the set of all curves modeled on $c$ can be given the quotient topology,
and the set of all curves the topology of a disjoint union as $c$ varies.

\begin{prop}\label{gencurves:prop}
In the space of all properly immersed curves in $\Sigma$ the subspace of the generic ones is open and dense.
\end{prop}

\begin{proof}
Let $C$ be the image of $j:c\to\Sigma$. Then, after an arbitrarily $\textrm{C}^{\infty}$ small perturbation, the image of $j$
satisfies the genericity conditions. And if $C$ is already generic, after any suitably $\textrm{C}^{\infty}$ small perturbation, the image of $j$ is still generic.
\end{proof}

We now define $\Gencurves$ as the set of all properly immersed curves in $\Sigma$ that are generic for $F$, up to diffeomorphisms of $\Sigma$ mapping $F$ to $F$.

\begin{defn}\label{gengraph:defn}
\emph{Let $G\in\calG(\Genflows)$ be the graph that defines $F$ on $\Sigma$ via $\varphi(\Genflows)$.
We define $\calG(\Gencurves)$ as the set of finite connected graphs $\Gamma$ obtained by adding to $G$ some bivalent vertices and a decoration
given by the following objects and subject to the following restrictions. Objects:
\begin{itemize}
\item A \emph{weight} in $\matN$ for every edge of $\Gamma$;
\item A \emph{height} in $\matN$ or in $\{+,-\}$ for every bivalent vertex.
\end{itemize}
Restrictions:
\begin{itemize}
\item At each trivalent vertex, the weight of the edge with label $A$ equals the sum of the weights of the edges with labels $T$ and $B$;
\item The weight of an edge with a univalent end is $0$;
\item If two edges with weights $n$ and $m$ share a bivalent vertex with height $\ell$, only the following can happen:
\begin{itemize}
\item[(2)] $m=n+2$, $1\leqslant \ell\leqslant n+1$.
\item[(3)] $m=n$ and  $n\geqslant2$, $\ell\in\matN$, $1\leqslant \ell\leqslant n-1$;
\item[(4)] $m=n+1$, $\ell=\pm$.
\end{itemize}
\end{itemize}
Such a $\Gamma$ is viewed up to homeomorphisms respecting the decoration (including the $T/B/A$ labels
of the germs of edge at the trivalent vertices).}
\end{defn}

We now describe the construction that associates to each element of $\calG(\Gencurves)$ a properly immersed curve in $\Sigma$ generic for $F$.

\begin{defn}
\emph{$\varphi(\Gencurves):\calG(\Gencurves)\to \Gencurves$ is
the map that associates to a graph $\Gamma$ the curve $C$ as follows.
If $\Gamma$ has no vertices, so it is a circle
of weight $n$, then $\Sigma$ is $S^1\times[0,1]$ with $F$ is parallel to $[0,1]$,
and $C$ consists of $n$ parallel copies of $S^1$.
Otherwise we cut each edge of $\Gamma$ into halves by means of a small dash and we enhance
the construction that associates $F$ on $\Sigma$ to $G$, shown in Fig.~\ref{flowportions:fig}, by adding to it the curve $C$, as shown in
Fig.~\ref{curveportions:fig}.
\begin{figure}
\faifig
{}
{\includegraphics[scale=0.6]{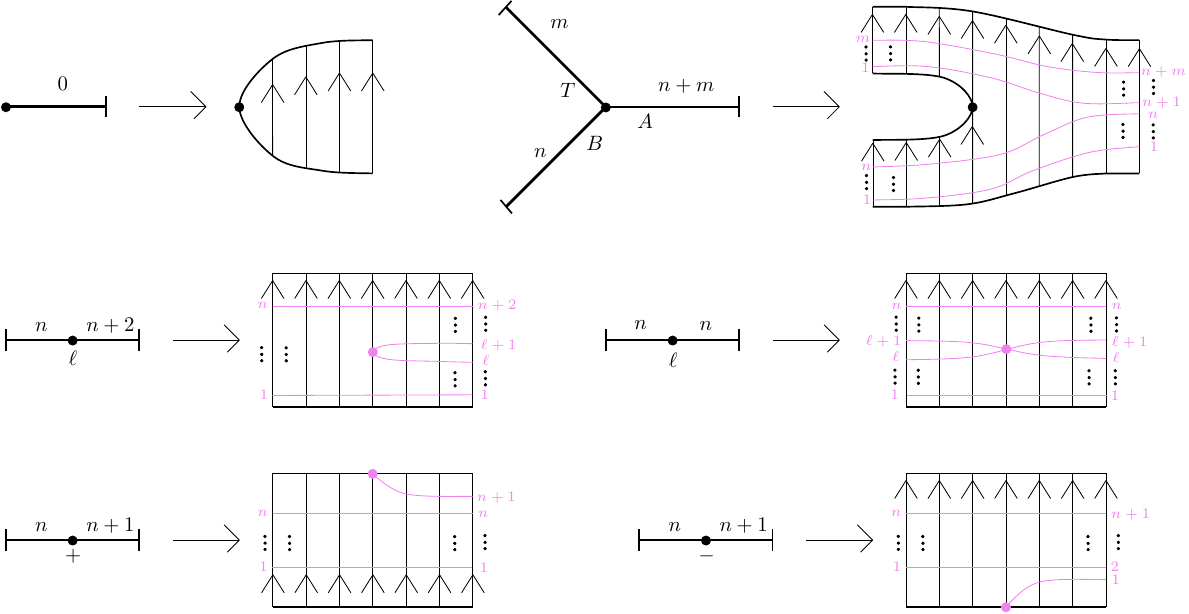}}
{Portions of a curve $C$ corresponding to portions of graph. All the points in these pictures are of type (0) if they do not belong to $C$,
and of type (1) if they do, except the following.
Top-left: a point-orbit; top-right: a point of type (5);
middle-left: a point of type (2); middle-right: a point of type (3);
bottom-left and bottom-right: a point of type (4).
\label{curveportions:fig}}
\end{figure}}
\end{defn}

\begin{rem}\label{contour:explanation:rem}
\emph{The following interpretation explains in what sense
the labelling of $\Gamma$ represents the \emph{apparent contour} of the associated $C$.
We can view the oriented orbits of $F$ as light rays entering $\Sigma$ through $\partial_+\Sigma$ and leaving it through
$\partial_-\Sigma=\partial\Sigma\setminus\partial_+\Sigma$, and we can imagine $C$ to be made of a semi-transparent material.
Note that the graph $G$ underlying $\Gamma$ can be viewed as $\partial_-\Sigma$ together
with the points at which $F$ is convexly tangent to $\partial\Sigma$. Now let an observer $O$ placed at a point $P$ of $\partial_-\Sigma$ look
at $C$ along the light ray $f\in F$ that ends at $P$. This is what $O$ sees:
\begin{itemize}
\item If $f$ contains only points of type (0) or (1) then $P$ is not a vertex of $\Gamma$ and $O$ will perceive
a light that gets dimmer for every stratum of $C$ the light ray $f$ has crossed. So $O$ can measure the intensity of the light reaching $P$, and
tell this number $n$ of strata. Along an edge of $\Gamma$ the value of $n$ does not change, and it is the weight of the edge;
\item If $f$ contains points of type (0) and (1) and one of type (2) then $P$ is a bivalent vertex, the weights of the edges incident to $P$ are
$n$ and $n+2$, and $O$ sees the point at which $C$ bends (precisely the apparent contour of $C$). Moreover $O$
can measure how dim the image of this point is, and hence tells how many strata of $C$ the light ray $f$ has crossed after it ($n-\ell+1$ in our notation,
if $P$ has label $\ell$);
\item If $f$ contains points of type (0) and (1) and one of type (3) then $P$ is a bivalent vertex, the weights of the edges incident to $P$ are both
$n$, and $O$ sees the point at which $C$ crosses itself, and tell how many strata of $C$ its image crosses before reaching $P$ ($n-\ell-1$ if
$P$ has label $\ell$);
\item If $f$ contains points of type (0) and (1) and one of type (4) then $P$ is a bivalent vertex, the weights of the edges incident to $P$ are
$n$ and $n+1$, and $O$ either directly sees a branch of $C$ coming to $P$, in which case $P$ has label $+$, or it sees (perhaps dimly, it does not
matter how much) a branch of $C$ reaching the first end of $f$, and $P$ has label $-$;
\item If $f$ contains points of type (0) and (1) and one of type (5) then $P$ corresponds in $G$ to a trivalent vertex, and the weights of the edges labeled $T$ and $B$ sum up
to give the weight of the edge labeled $A$.
\end{itemize}}
\end{rem}

\begin{rem}\label{graphs:for:embedded:closed:curves:rem}
\emph{$\Gamma\in\calG(\Gencurves)$ gives a properly \emph{embedded} curve if it has no bivalent vertex of type (3), a \emph{closed} immersed curve if it has
no bivalent vertex of type (4), and a \emph{closed embedded} curve if has no bivalent vertex of type (3) or (4).}
\end{rem}

Recall that $\pi(\Gencurves): \Gencurves\to \Curves$ is the natural projection (see the introduction).

\begin{prop}\label{curves:prop}\
\begin{itemize}
\item $\varphi(\Gencurves):\calG(\Gencurves)\to \Gencurves$ is bijective;
\item $\pi(\Gencurves): \Gencurves\to \Curves$ is surjective.
\end{itemize}
\end{prop}

\begin{proof}
For the first item, we have associated to $\Gamma$ a curve $C$ which is well-defined up to diffeomorphisms of $\Sigma$ preserving $F$,
and only depends on $\Gamma$ up to decorated homeomorphisms.  Moreover the construction can be reversed, starting
from a generic curve $C$ and getting a decorated graph $\Gamma$, just by reversing the arrows in Fig.~\ref{curveportions:fig}.
The second item follows from Proposition~\ref{gencurves:prop}.
\end{proof}

We conclude this section by proving the necessity of all the labels we have introduced on a graph $G\in\calG(\Genflows)$ encoding $F$ on $\Sigma$ to turn
it into a $\Gamma\in\calG(\Gencurves)$ encoding a curve $C$ in $\Sigma$ generic for $F$. In fact, it will be enough to do so for $\Sigma=S^1\times[0,1]$
and $F$ parallel to $[0,1]$, so $G$ is a circle. A weight $n$ on $G$ gives a curve with $n$ components, so $n$ is necessary.
Fig.~\ref{labelsnec:fig} shows that the other labels are necessary as well.
\begin{figure}
\faifig
{}
{\includegraphics[scale=0.6]{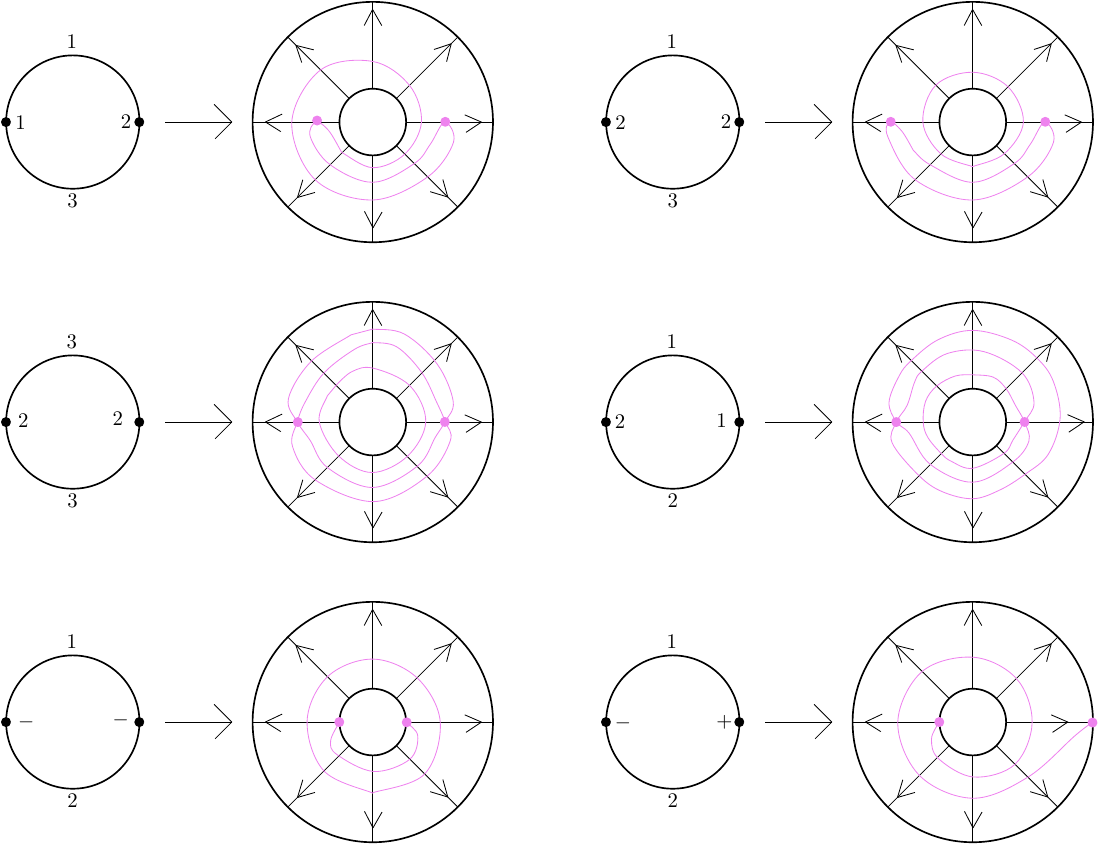}}
{Each of the three lines shows two graphs in $\calG(\Gencurves)$ that differ from each other for the label of one bivalent vertex,
and the associated curves. The difference is respectively at vertices giving points of type (2), (3) and (4).
\label{labelsnec:fig}}
\end{figure}


\section{Moves for a fixed flow}

In this section the setting will be the same as in the previous one, namely we fix $\Sigma$, a generic flow $F$ on $\Sigma$
and the flow-spine $G\in\calG(\Genflows)$ defining $F$ on $\Sigma$. Recall that $\calG(\Gencurves)$ is the set of graphs $\Gamma$
obtained by adding decorations to $G$ as in Definition~\ref{gengraph:defn}. We introduce on $\calG(\Gencurves)$
the set $\calM(\Gencurves)$ of the moves $i^*_*$ defined in
Figg.~\ref{D-moves:fig} to~\ref{4-5-moves:fig},
where the symbol over each $\leftrightarrow$ is the name of the move, and the formula under it is the condition that the involved weights and heights must
satisfy in order for the move to be applicable.

\begin{rem}\label{universal:moves:rem}
\emph{The moves in $\calM(\Gencurves)$ do not actually depend on $\Sigma$ and $F$: all of them are finite local combinatorial moves,
that would be the same for any $\Sigma$ and $F$.}
\end{rem}

\begin{figure}
\faifig{}
{\includegraphics[scale=0.6]{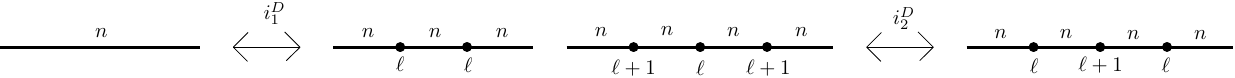}}
{The moves $i^D_j$.\label{D-moves:fig}}
\end{figure}

\begin{figure}
\faifig{}
{\includegraphics[scale=0.6]{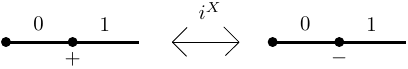}}
{The move $i^X$.\label{X-move:fig}}
\end{figure}

\begin{figure}
\faifig{}
{\includegraphics[scale=0.6]{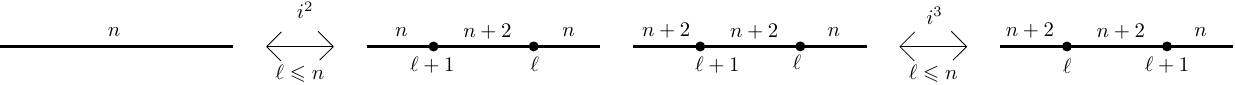}}
{The moves $i^2$ and $i^3$.\label{2-move+3-move:fig}}
\end{figure}

\begin{figure}
\faifig{}
{\includegraphics[scale=0.6]{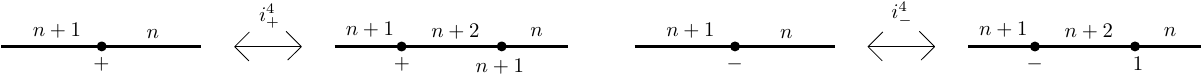}}
{The move $i^4_{\pm}$.\label{4-moves:fig}}
\end{figure}

\begin{figure}
\faifig{}
{\includegraphics[scale=0.6]{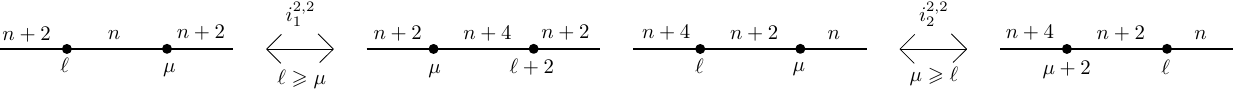}}
{The moves $i^{2,2}_j$.\label{2-2-moves:fig}}
\end{figure}

\begin{figure}
\faifig{}
{\includegraphics[scale=0.6]{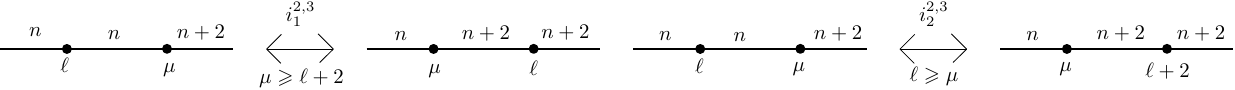}}
{The moves $i^{2,3}_j$.\label{2-3-moves:fig}}
\end{figure}

\begin{figure}
\faifig{}
{\includegraphics[scale=0.6]{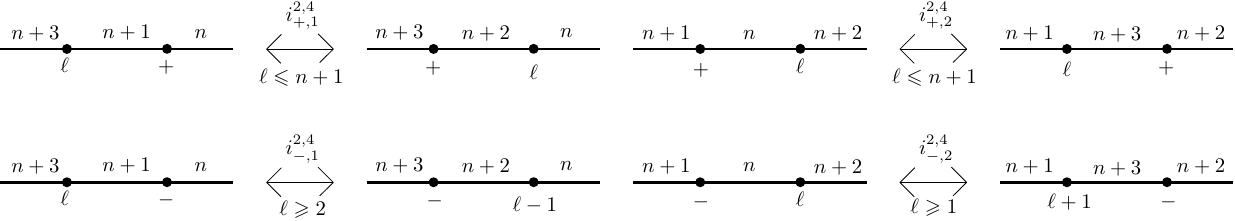}}
{The moves $i^{2,4}_{\pm,j}$.\label{2-4-moves:fig}}
\end{figure}

\begin{figure}
\faifig{}
{\includegraphics[scale=0.6]{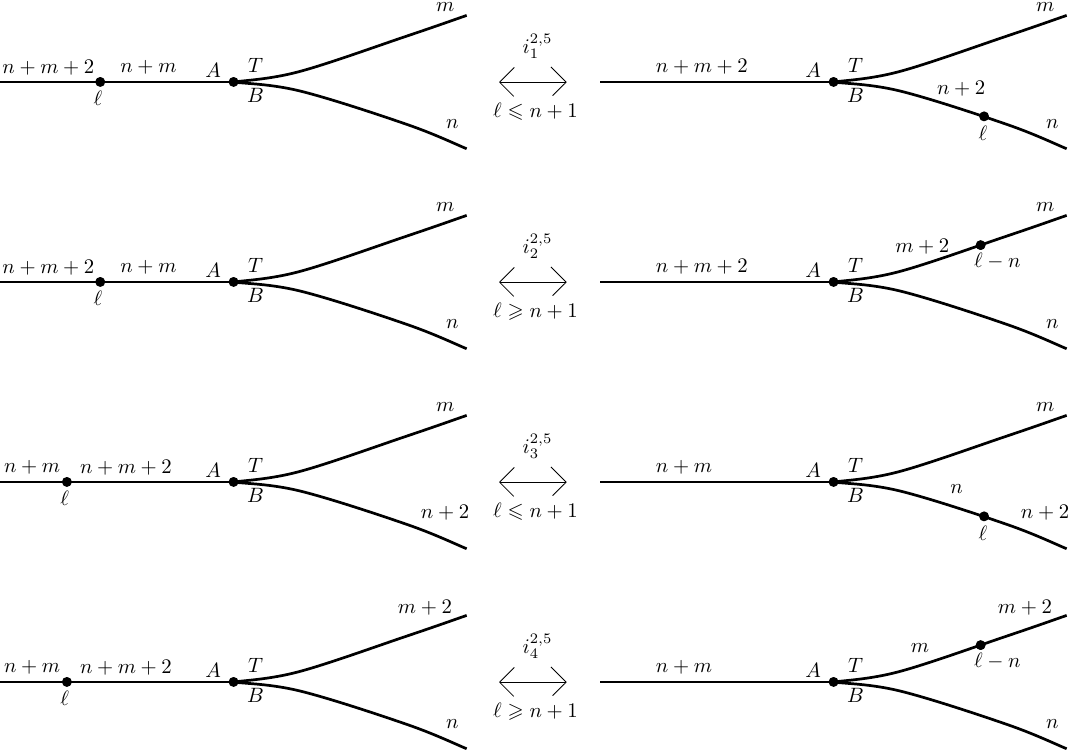}}
{The moves $i^{2,5}_j$.\label{2-5-moves:fig}}
\end{figure}

\begin{figure}
\faifig{}
{\includegraphics[scale=0.6]{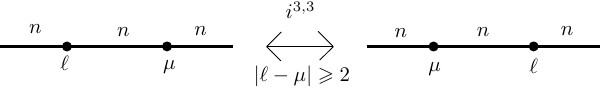}}
{The move $i^{3,3}$.\label{3-3-move:fig}}
\end{figure}

\begin{figure}
\faifig{}
{\includegraphics[scale=0.6]{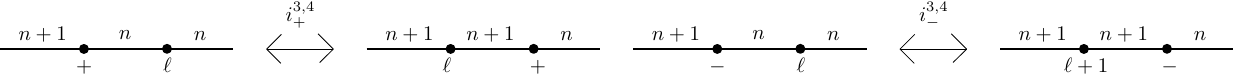}}
{The moves $i^{3,4}_\pm$.\label{3-4-moves:fig}}
\end{figure}

\begin{figure}
\faifig{}
{\includegraphics[scale=0.6]{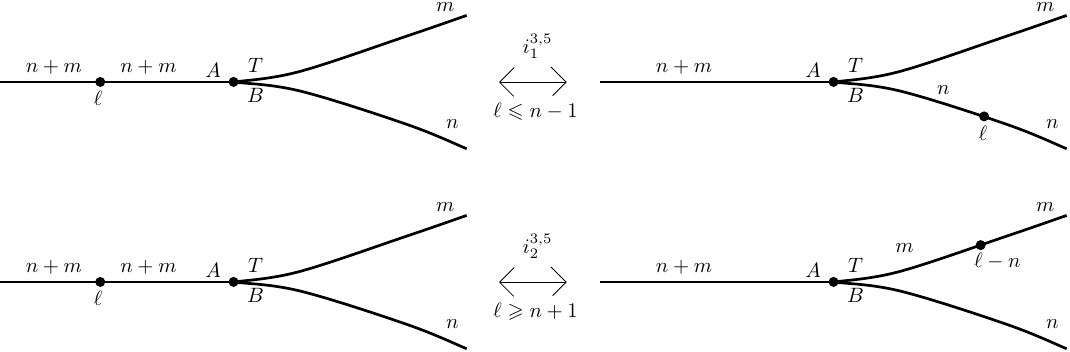}}
{The moves $i^{3,5}_j$.\label{3-5-moves:fig}}
\end{figure}

\begin{figure}
\faifig{}
{\includegraphics[scale=0.6]{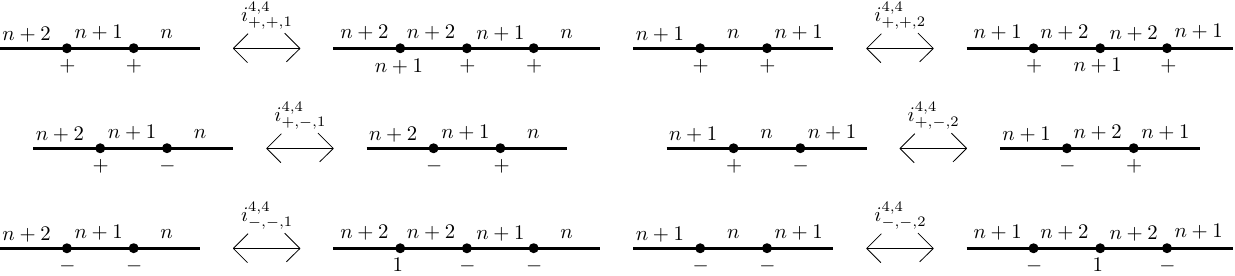}}
{The moves $i^{4,4}_{\pm,\pm,j}$.\label{4-4-moves:fig}}
\end{figure}

\begin{figure}
\faifig{}
{\includegraphics[scale=0.6]{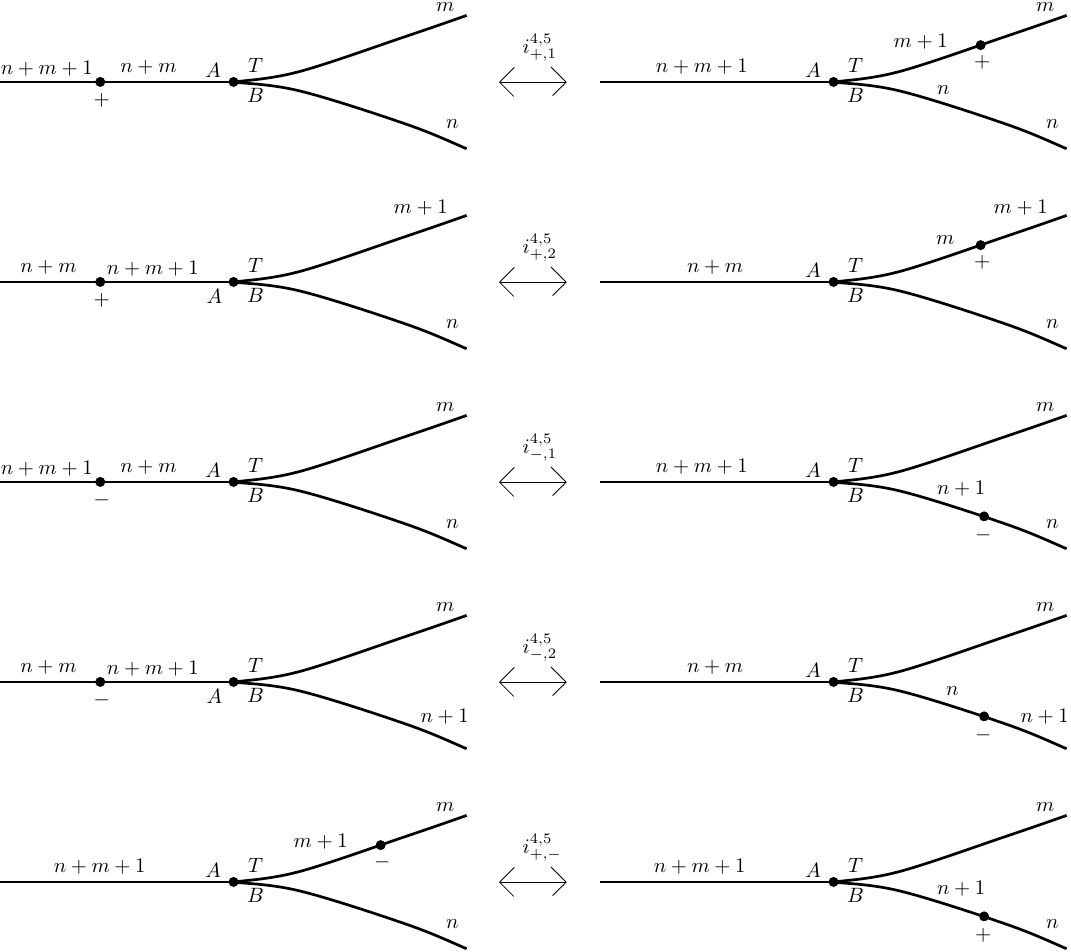}}
{The moves $i^{4,5}_{\pm,j}$ and $i^{4,5}_{+,-}$.\label{4-5-moves:fig}}
\end{figure}

\begin{rem}\label{left-right:rem}
\emph{All our moves are described as $\texttt{L}\mathop{\leftrightarrow}\limits^i\texttt{R}$ where \texttt{L} and \texttt{R} are labelled planar diagrams.
Now let $\reflectbox{\texttt{L}}$ and $\reflectbox{\texttt{R}}$ be the same diagrams vertically mirrored.
Then $\texttt{L}\mathop{\leftrightarrow}\limits^i\texttt{R}$ stands for any of the moves
\begin{center}
$\texttt{L}\to\texttt{R}$ \qquad
$\texttt{R}\to\texttt{L}$ \qquad
$\reflectbox{\texttt{L}}\to\reflectbox{\texttt{R}}$ \qquad
$\reflectbox{\texttt{R}}\to\reflectbox{\texttt{L}}$.
\end{center}
Note in particular that if $\reflectbox{\texttt{L}}=\texttt{L}$ but $\reflectbox{\texttt{R}}\neq \texttt{R}$ there are two ways to apply $i$ to \texttt{L},
and similarly if $\reflectbox{\texttt{R}}=\texttt{R}$ but $\reflectbox{\texttt{L}}\neq \texttt{L}$.
Observe also that the truth of $\reflectbox{\texttt{L}}=\texttt{L}$ and $\reflectbox{\texttt{R}}=\texttt{R}$
can depend on the labels (for $i^{2,2}_1$, for instance, we have $\reflectbox{\texttt{L}}=\texttt{L}$, $\reflectbox{\texttt{R}}\neq \texttt{R}$ for $\mu=\ell$,
while $\reflectbox{\texttt{L}}\neq \texttt{L}$, $\reflectbox{\texttt{R}}=\texttt{R}$ for $\mu=\ell+2$, and
$\reflectbox{\texttt{L}}\neq \texttt{L}$, $\reflectbox{\texttt{R}}\neq \texttt{R}$ otherwise).}
\end{rem}

\begin{thm}\label{curves:thm}\
\begin{itemize}
\item $\varphi(\Gencurves):\calG(\Gencurves)\to \Gencurves$ is bijective;
\item $\pi(\Gencurves): \Gencurves\to \Curves$ is surjective;
\item Two graphs in $\calG(\Gencurves)$ have the same image in $\Curves$ under the composition $\pi(\Gencurves)\compo\varphi(\Gencurves)$
if and only if they are related by a finite combination of moves in $\calM(\Gencurves)$.
\end{itemize}
\end{thm}

\begin{proof}
The first two items were already shown in Proposition~\ref{curves:prop}.

For the third item we must list all the possible catastrophes, namely the elementary violations of the genericity condition of Definition~\ref{gencurve:defn},
that can occur along a homotopy of $C$ and that cannot be avoided by a small perturbation, and show that they translate into the moves in $\calM(\Gencurves)$.
These catastrophes come in two flavours:
\begin{itemize}
\item[(I)] Elementary violations of one of the three conditions in (D) or of the single condition in (X), (2), (3), or (4);
\item[(II)] Presence on the same arc orbit $f$ of $F$ of two points of type (2), (3), (4), or (5).
\end{itemize}

We will now describe these catastrophes in detail and at the same time provide their diagrammatic translation.
Starting with type (I), an elementary violation of (D) occurs if there is a point $P$ as follows:
\begin{itemize}

\item[(D$_1$)] A non-transverse double point of $C$ away from $\partial\Sigma$; we can assume
that the two strands of $C$ through $P$ have order-2 contact with each other and are transverse to $F$ at $P$; then
Fig.~\ref{D-move-1-proof:fig} shows that this corresponds to the move $i^{D}_1$
(here, and always later, in direction left to right);
    \begin{figure}
    \faifig{}
    {\includegraphics[scale=0.6]{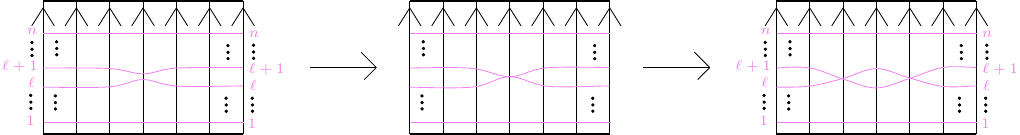}}
    {The catastrophe giving $i^D_1$.\label{D-move-1-proof:fig}}
    \end{figure}

\item[(D$_2$)] A transverse triple point of $C$ away from $\partial\Sigma$; we can assume
that all three strands of $C$ through $P$ are transverse to $F$ at $P$; Fig.~\ref{D-move-2-proof:fig} shows that this corresponds to the move $i^{D}_2$;
    \begin{figure}
    \faifig{}
    {\includegraphics[scale=0.6]{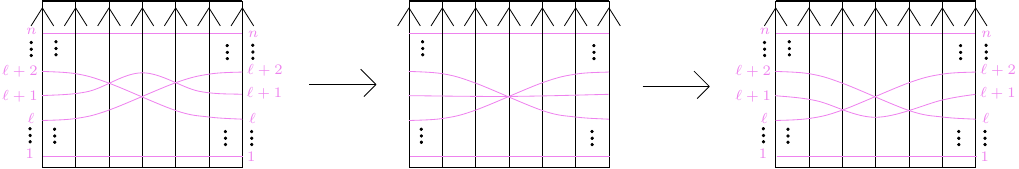}}
    {The catastrophe giving $i^D_2$.\label{D-move-2-proof:fig}}
    \end{figure}

\item A transverse double point on $\partial\Sigma$; we can assume that both the strands of $C$ through $P$ are transverse to $F$ at $P$;
this case can also be described as one in which a single orbit of $F$ contains two points of type (4), so we can treat it later.

\end{itemize}

The other catastrophes of type (I) come if there is a point $P$ as follows:

\begin{itemize}

\item[(X)] An orbit of $F$ reduces to $P$ and $P$ is the end of a strand of $C$;
Fig.~\ref{X-move-proof:fig} shows we get $i^{X}$;
    \begin{figure}
        \faifig{}
    {\includegraphics[scale=0.6]{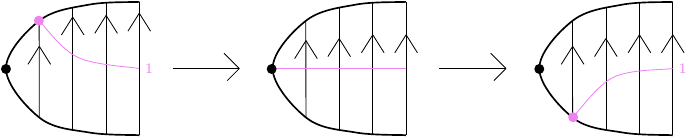}}
    {The catastrophe giving $i^X$.\label{X-move-proof:fig}}
    \end{figure}

\item[(2)] $P$ is a  point of triple contact between $C$ and an orbit of $F$; by
Fig.~\ref{2-move-proof:fig} we get $i^{2}$;
    \begin{figure}
        \faifig{}
    {\includegraphics[scale=0.6]{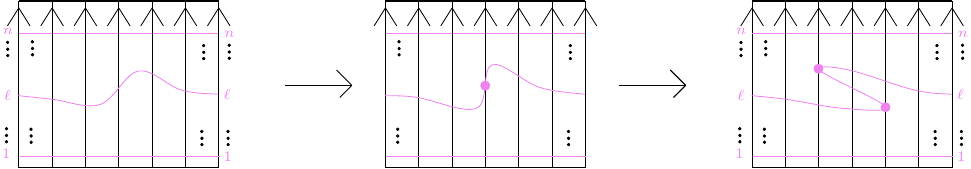}}
    {The catastrophe giving $i^2$.\label{2-move-proof:fig}}
    \end{figure}

\item[(3)] $P\in\calD(C)$ and one of the strands of $C$ at $P$ is tangent
to an orbit of $F$; we can assume the order of contact is $2$;
by Fig.~\ref{3-move-proof:fig} we get $i^{3}$;
    \begin{figure}
    \faifig{}
    {\includegraphics[scale=0.6]{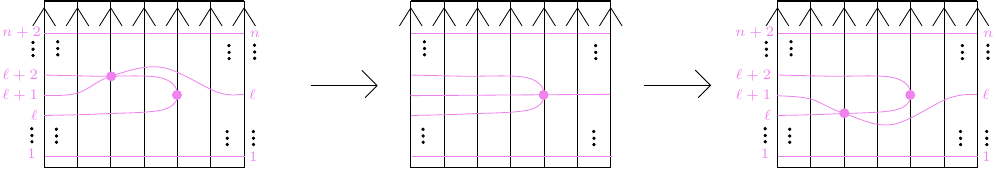}}
    {The catastrophe giving $i^3$.\label{3-move-proof:fig}}
    \end{figure}

\item[(4)] $P\in C\cap \partial\Sigma$ and $C$ is tangent to the orbit of $F$ through $P$;
by Fig.~\ref{4-moves-proof:fig}, depending on whether $P\in\partial_+\Sigma$ or $P\in\partial\Sigma\setminus\partial_+\Sigma$, we get $i^{4}_\pm$.
    \begin{figure}
    \faifig{}
    {\includegraphics[scale=0.6]{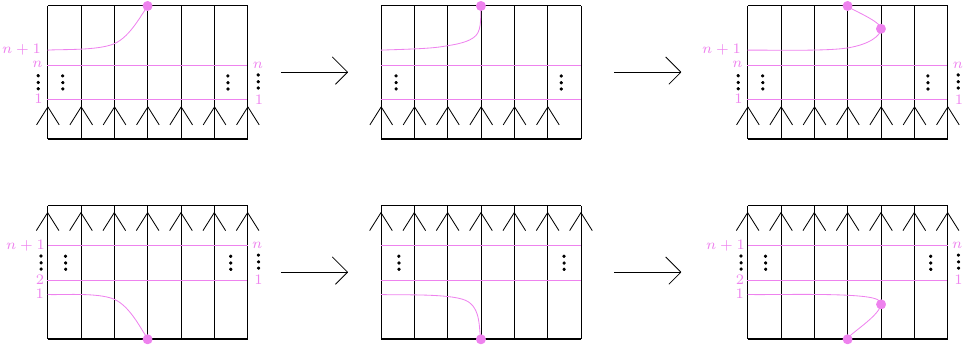}}
    {The catastrophes giving $i^4_{\pm}$.\label{4-moves-proof:fig}}
    \end{figure}
\end{itemize}

Now we turn to catastrophes ot type (II), that occur if there is an orbit $f$ of $F$ as follows:

\begin{itemize}

\item[(2,2)] Two branches of $C$ are tangent to $f$ (at different points and both with an order-2 contact);
by Fig.~\ref{2-2-moves-proof:fig}, depending on whether these branches lie to opposite sides of $f$ or to the same side, we get $i^{2,2}_j$ for $j=1,2$;
    \begin{figure}
    \faifig{}
    {\includegraphics[scale=0.6]{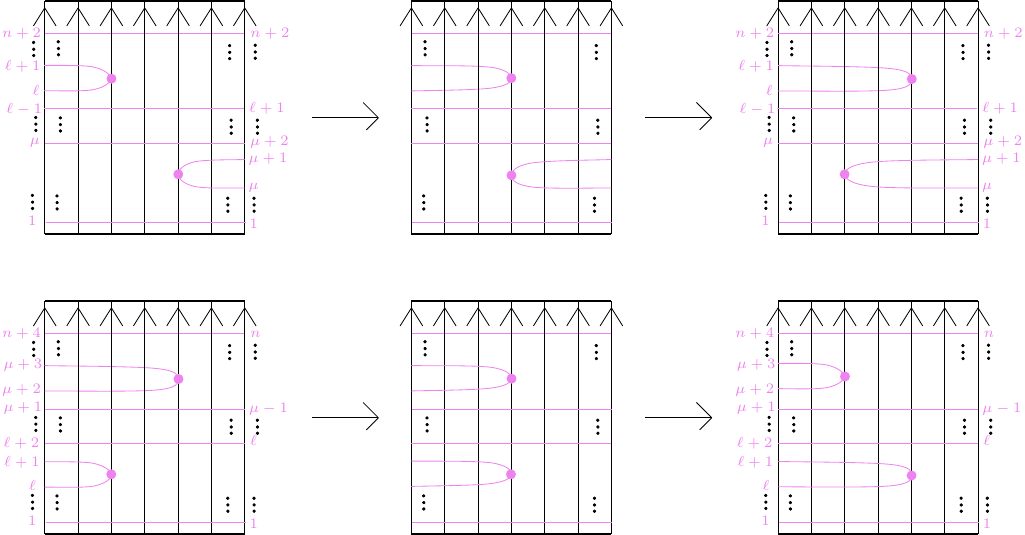}}
    {The catastrophes giving $i^{2,2}_j$.\label{2-2-moves-proof:fig}}
    \end{figure}

\item[(2,3)] A branch of $C$ is tangent to $f$ (with an order-2 contact) at a point $P$ and two other branches cross each other transversely at a point $Q$ of $f$
(and they cross $f$ transversely);
by Fig.~\ref{2-3-moves-proof:fig}, depending on whether $P$ comes after or before $Q$ along $f$, we get $i^{2,3}_j$ for $j=1,2$;
    \begin{figure}
    \faifig{}
    {\includegraphics[scale=0.6]{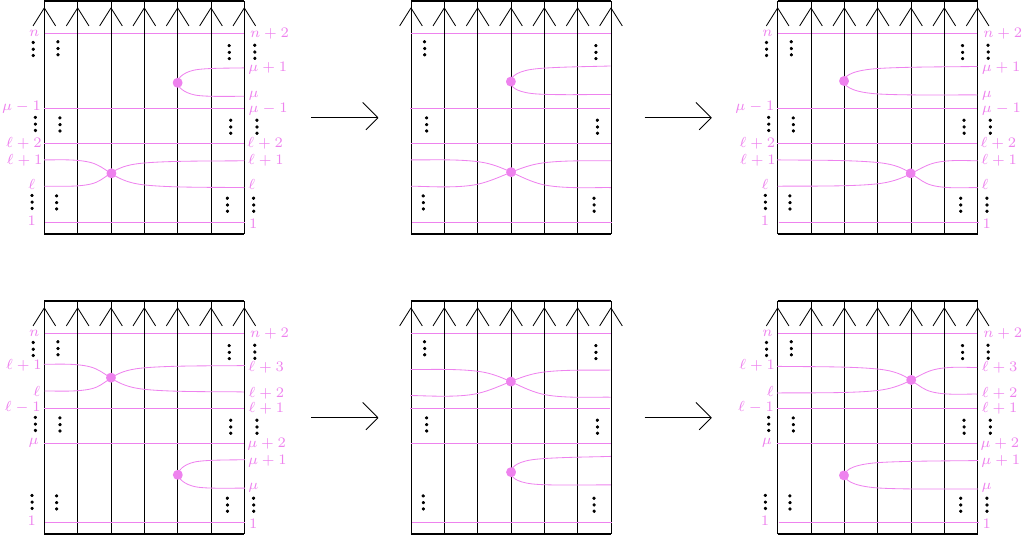}}
    {The catastrophes giving $i^{2,3}_j$.\label{2-3-moves-proof:fig}}
    \end{figure}

\item[(2,4)] A branch of $C$ is tangent to $f$ (with an order-2 contact) and another branch of $C$ meets $\partial\Sigma$ at an end $P$ of $f$;
by Fig.~\ref{2-4-moves-proof:fig}, depending on whether $P$ is the first or last end of $f$ and on whether the two branches of $C$ lie
to the same or to opposite sides of $f$, we get $i^{2,4}_{\pm,j}$ for $j=1,2$;
    \begin{figure}
    \faifig{}
    {\includegraphics[scale=0.6]{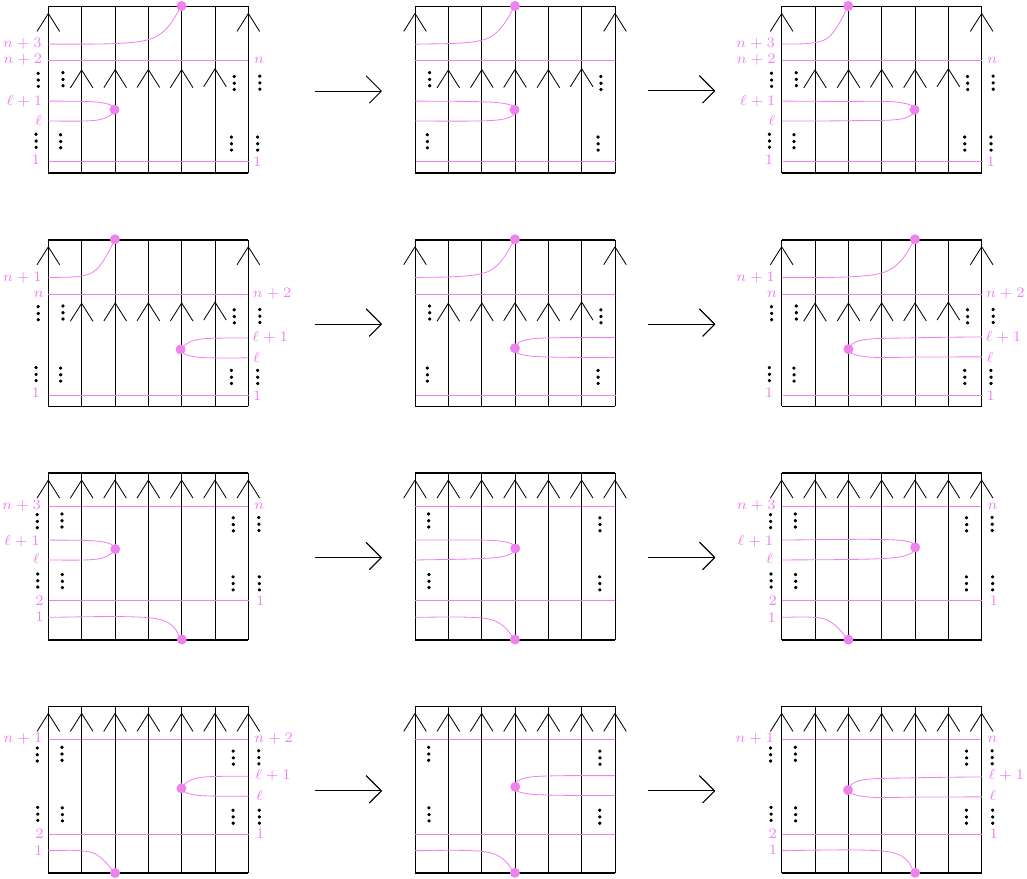}}
    {The catastrophes giving $i^{2,4}_{\pm,j}$.\label{2-4-moves-proof:fig}}
    \end{figure}

\item[(2,5)] A branch of $C$ and one of $\partial\Sigma$ are tangent to $f$ (with an order-2 contact)
at points $P$ and $Q$ respectively;
by Fig.~\ref{2-5-moves-proof:fig}, depending on whether
the branches of $C$ and $\partial\Sigma$ lie to opposite or to the same side of $f$
and on whether
$P$ comes before or after $Q$ on $f$, we get $i^{2,5}_j$ for $j=1,\ldots,4$;
    \begin{figure}
    \faifig{}
    {\includegraphics[scale=0.6]{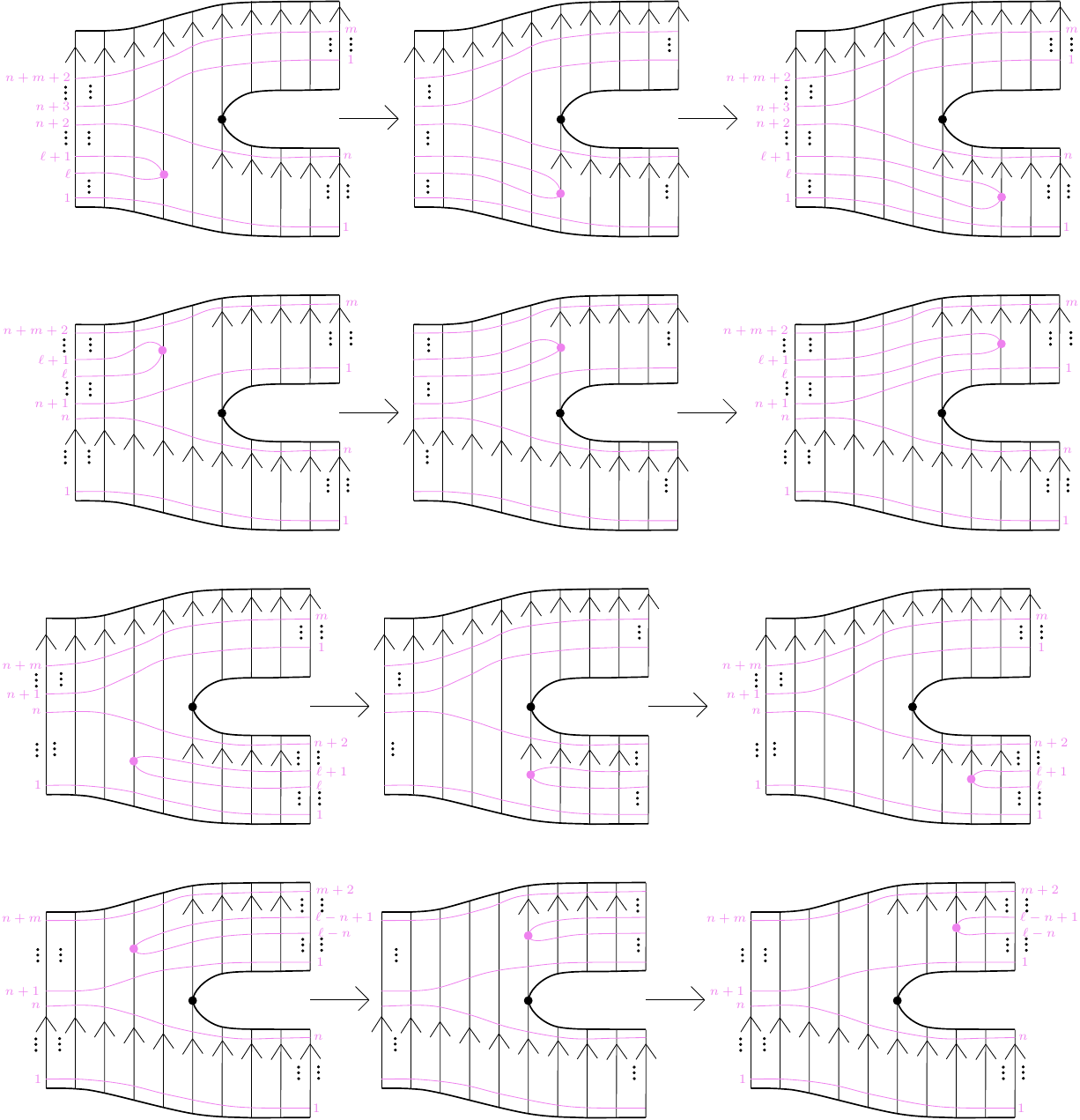}}
    {The catastrophes giving $i^{2,5}_{j}$.\label{2-5-moves-proof:fig}}
    \end{figure}

\item[(3,3)] Two transverse double points of $C$ belong to $f$ (and all the branches of $C$ are transverse to $f$);
by Fig.~\ref{3-3-move-proof:fig} we get $i^{3,3}$;
    \begin{figure}
    \faifig{}
    {\includegraphics[scale=0.6]{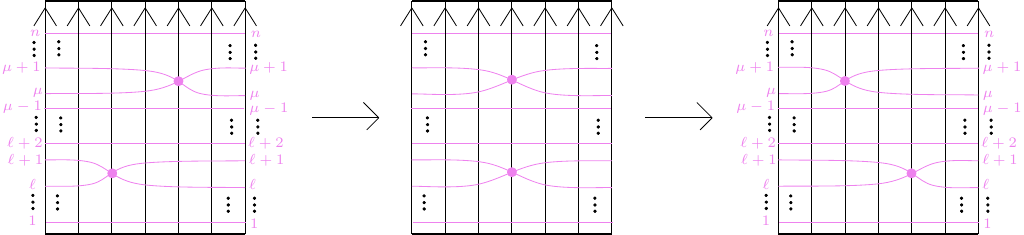}}
    {The catastrophe giving $i^{3,3}$.\label{3-3-move-proof:fig}}
    \end{figure}

\item[(3,4)] A transverse double point of $C$ and an endpoint $P$ of $C$ belong to $f$  (and all the branches of $C$ are transverse to $f$);
by Fig.~\ref{3-4-moves-proof:fig}, depending on whether $P$ is the last or first end of $f$, we get $i^{3,4}_\pm$;
\begin{figure}
    \faifig{}
    {\includegraphics[scale=0.6]{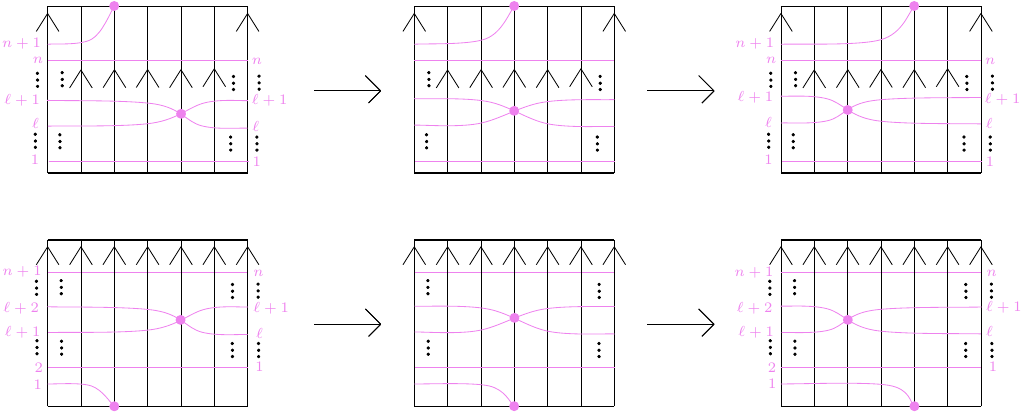}}
    {The catastrophes giving $i^{3,4}_{\pm}$.\label{3-4-moves-proof:fig}}
    \end{figure}

\item[(3,5)] A transverse double point $P$ of $C$ belongs to $f$ (and all the branches of $C$ are transverse to $f$) and a
branch of $\partial\Sigma$ is tangent to $f$ (with an order-2 contact) at a point $Q$;
by Fig.~\ref{3-5-moves-proof:fig}, depending on whether $P$ comes after or before $Q$ on $f$, we get $i^{3,5}_j$ for $j=1,2$;
    \begin{figure}
    \faifig{}
    {\includegraphics[scale=0.6]{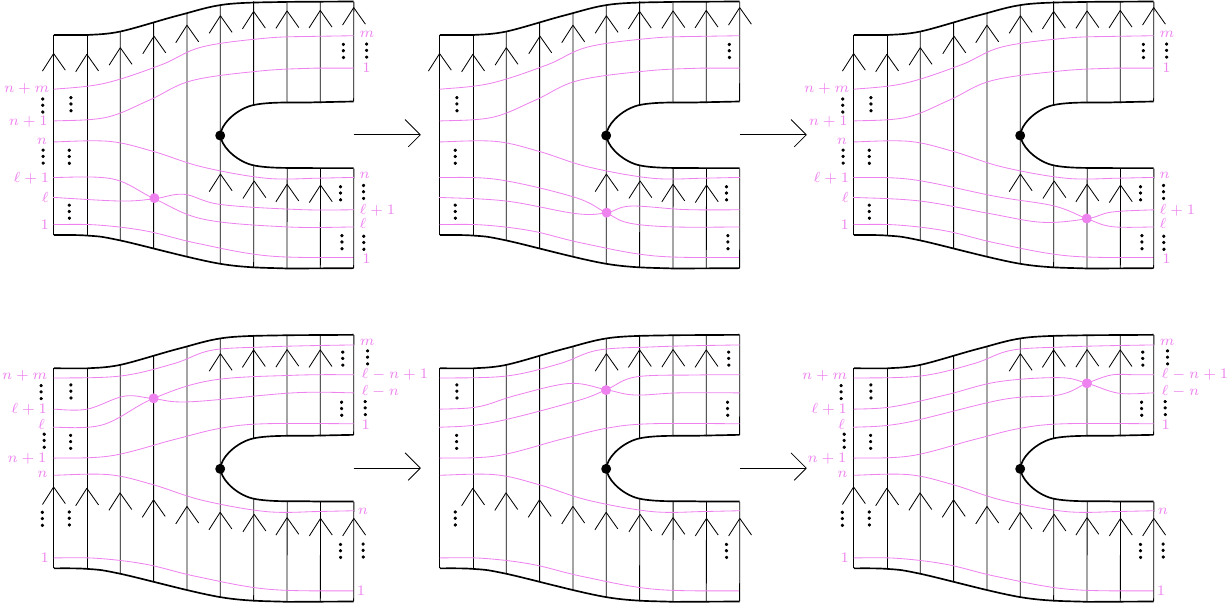}}
    {The catastrophes giving $i^{3,5}_{j}$.\label{3-5-moves-proof:fig}}
    \end{figure}

\item[(4,4)] Two branches of $C$ reach $\partial \Sigma$ at an end of $f$ (transversely to $\partial\Sigma$ and to $f$, and to each other of they meet);
here we must distinguish several cases: the two branches of $C$ can both reach the second end of $f$, or the different ends of $f$, or both
the first end of $f$; and they can lie to the same side or to opposite sides of $f$; then,
by Fig.~\ref{4-4-moves-proof:fig}, we get $i^{4,4}_{\pm,\pm,j}$ for $j=1,2$;
    \begin{figure}
    \faifig{}
    {\includegraphics[scale=0.6]{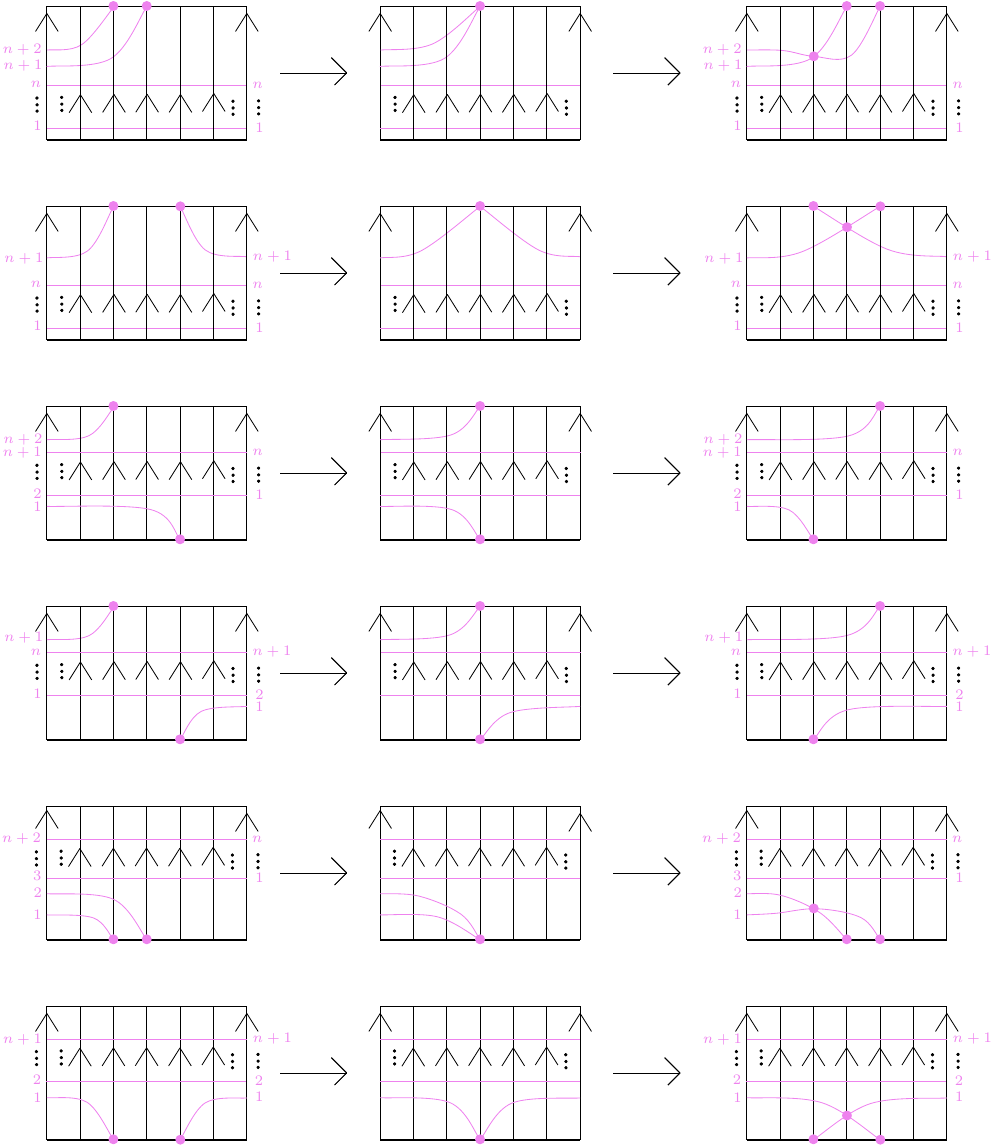}}
    {The catastrophes giving $i^{4,4}_{\pm,\pm,j}$.\label{4-4-moves-proof:fig}}
    \end{figure}

\item[(4,5)] A branch of $C$ reaches $\partial\Sigma$ at a point of $f$
(transversely to both $\partial\Sigma$ and $f$) and a branch of $\partial \Sigma$ is tangent to $f$ (with an order-2 contact);
by Fig.~\ref{4-5-moves-proof:fig}, depending on whether $C$ reaches the second end of $f$, its first end or the tangency point to $\partial\Sigma$,
and (in the first two cases) on whether the involved branches of
$C$ and $\partial\Sigma$ lie to opposite or to the same side of $f$, we get $i^{4,5}_{\pm,j}$ for $j=1,2$ and $i^{4,5}_{+,-}$.
    \begin{figure}
    \faifig{}
    {\includegraphics[scale=0.6]{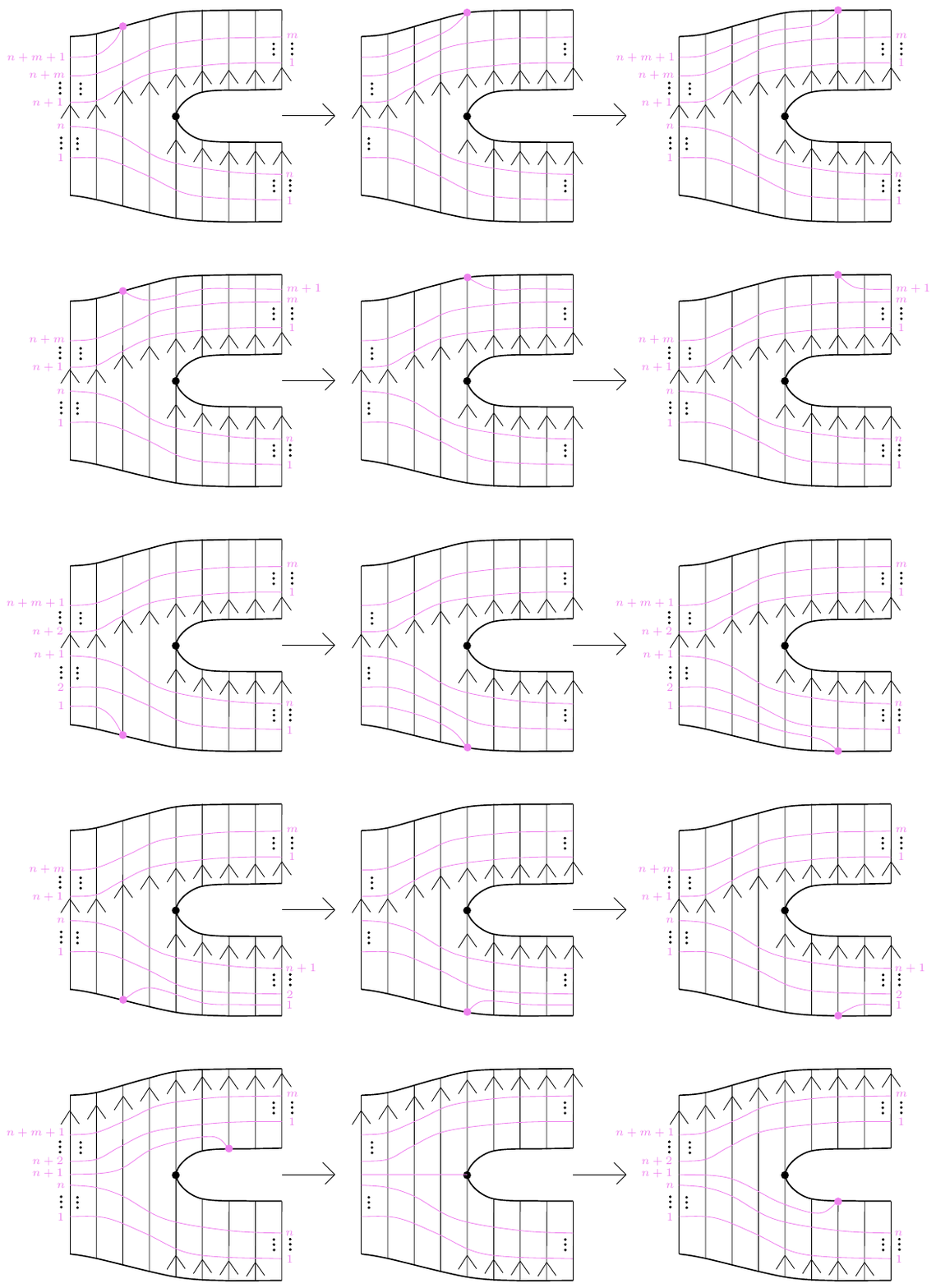}}
    {The catastrophes giving $i^{4,5}_{\pm,j}$ and $i^{4,5}_{+,-}$.\label{4-5-moves-proof:fig}}
    \end{figure}

\end{itemize}

Our argument is complete.
\end{proof}

\begin{rem}\label{moves:for:embedded:closed:curves:rem}
\emph{Theorem~\ref{curves:thm} provides a combinatorial presentation of the set of properly immersed curves in $\Sigma$ up to homotopy through similar curves.
In the spirit of Remark~\ref{graphs:for:embedded:closed:curves:rem}, this presentation restricts to
one of properly \emph{embedded} curves in $\Sigma$, of immersed \emph{closed} curves in $\Sigma$, or of \emph{embedded closed} curves in $\Sigma$ by
forbidding bivalent vertices of some type in Definition~\ref{gengraph:defn} and ignoring the moves involving them. Namely, respectively:
\begin{itemize}
\item Forbid type (3) and ignore
$i^D_j$, $i^3$, $i^{2,3}_j$, $i^{3,3}$, $i^{3,4}_\pm$, $i^{3,5}_j$;
\item Forbid type (4) and ignore
$i_X$, $i^4_{\pm}$, $i^{2,4}_{\pm,j}$, $i^{3,4}_\pm$, $i^{4,4}_{\pm,\pm,j}$, $i^{4,5}_{\pm,j}$, $i^{4,5}_{+,-}$;
\item Forbid types (3) and (4) and ignore all the previously listed moves.
\end{itemize}}
\end{rem}

\section{Simultaneous homotopy\\ of the curve and the flow}
In this section we will put together the results of the previous ones.
Recall that $\Pairs$ is the quotient of the set of
all pairs $(F,C)$, where $F$ is a traversing flow and $C$ is a properly immersed curve in the same (varying) surface $\Sigma$,
up to diffeomorphisms of $\Sigma$ and simultaneous homotopic variation of $F$ and $C$.
We then define $\Genpairs$ as the set of all such pairs $(F,C)$, with $F$ generic and $C$ generic for $F$,
up to diffeomorphisms of $\Sigma$, and $\pi(\Genpairs):\Genpairs\to\Pairs$ as the obvious quotient map.

We now define $\calG(\Genpairs)$ as the set of all finite connected graphs with vertices of valence $1$, $2$ or $3$,
a $T/B/A$ labelling of the germ of edges at each trivalent vertex, and a decoration of the edges and the bivalent vertices precisely as in
Definition~\ref{gengraph:defn}.
In other words, $\calG(\Genpairs)$ is the disjoint union of all $\calG(\Gencurves)$ as $\Sigma$ varies among surfaces and $F$ varies
among generic traversing flows on $\Sigma$. We then have an obvious bijective reconstruction map
$$\varphi(\Genpairs):\calG(\Genpairs)\to\Genpairs.$$

Next, we introduce on $\calG(\Genpairs)$ the set of finite local combinatorial moves $\calM(\Genpairs)$ given by
those defined in
Fig.~\ref{i-from-s-moves:fig}
together with all those in
$\calM(\Gencurves)$, see Figg.~\ref{D-moves:fig} to~\ref{4-5-moves:fig} and recall Remark~\ref{universal:moves:rem}.
\begin{figure}
\faifig{}
{\includegraphics[scale=0.6]{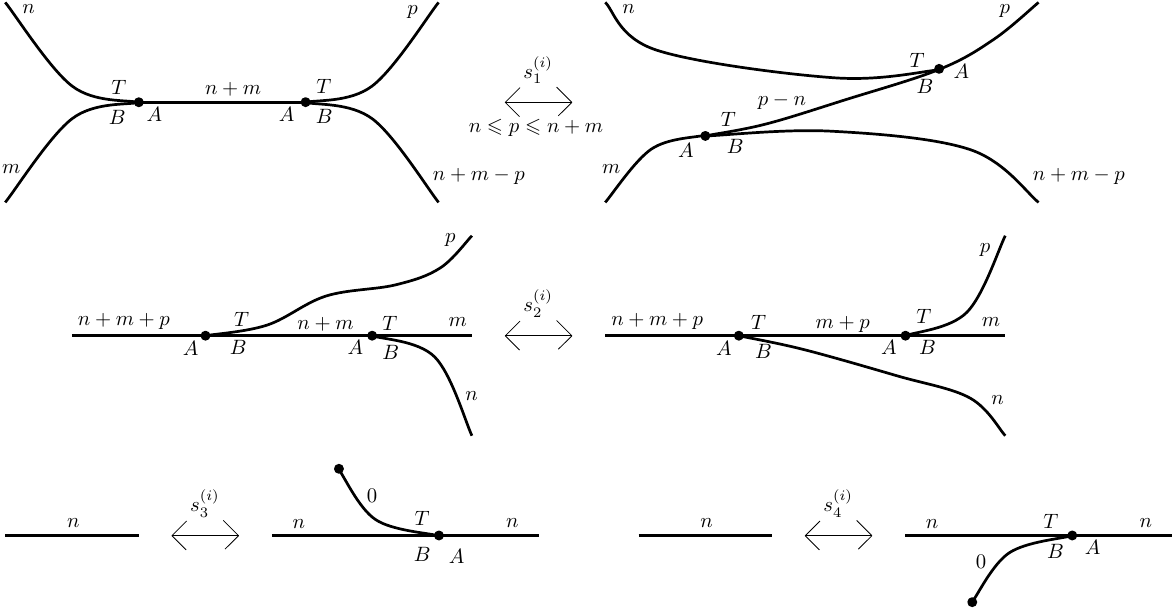}}
{Moves for graphs representing properly immersed curves derived from the moves for flow-spines of
Fig.~\ref{s-moves:fig}\label{i-from-s-moves:fig}}
\end{figure}

\begin{rem}\label{left-right-bis:rem}
\emph{The discussion of Remark~\ref{left-right:rem} applies to the moves of Fig.~\ref{i-from-s-moves:fig}.
Here we always have
$\reflectbox{\texttt{R}}\neq\texttt{R}$, while
$\reflectbox{\texttt{L}}=\texttt{L}$ precisely for $s^{(i)}_1$ with $p=n$ (\emph{i.e.}, $n+m-p=m$), for
$s^{(i)}_3$ and for $s^{(i)}_4$, so in all these cases we have both the moves
$\texttt{L}\to\texttt{R}$ and
$\texttt{L}\to\reflectbox{\texttt{R}}$.}
\end{rem}

\begin{thm}\label{pairs:thm}\
\begin{itemize}
\item $\varphi(\Genpairs):\calG(\Genpairs)\to \Genpairs$ is bijective;
\item $\pi(\Genpairs): \Genpairs\to \Pairs$ is surjective;
\item Two graphs in $\calG(\Genpairs)$ have the same image in $\Pairs$ under the composition $\pi(\Genpairs)\compo\varphi(\Genpairs)$
if and only if they are related by a finite combination of moves in $\calM(\Genpairs)$.
\end{itemize}
\end{thm}

\begin{proof}
The first two items were already discussed above.

For the third item, we remark that a simultaneous homotopy of a traversing flow $F$ and a properly immersed curve $C$ in $\Sigma$ can be perturbed so that:
\begin{itemize}
\item At all but finitely many times the flow $F$ is generic and the curve $C$ is generic for $F$;
\item At each of the special times:
\begin{itemize}
\item Either $F$ is generic and the curve $C$ undergoes one of the catastrophes already examined in the proof of
Theorem~\ref{curves:thm};
\item Or $F$ undergoes one of the catastrophes examined in the proof Theorem~\ref{flows:thm}, namely there is an orbit
$f$ of $F$ with two order-$2$ or one order-$3$ tangency to $\partial\Sigma$, but
$f$ is transverse to $C$ and the point(s) of $f\cap\partial\Sigma$ do not belong to $C$.
\end{itemize}
\end{itemize}

The transition through one of these times then corresponds to either one of the moves in $\calM(\Gencurves)$
or, as shown in Figg.~\ref{si-1-move-proof:fig} to Fig.~\ref{si-3-move-proof:fig}, to one of the extra moves
of Fig.~\ref{i-from-s-moves:fig} (the figure for $s_4^{(i)}$ is very similar to that for $s_3^{(i)}$).
\begin{figure}
\faifig{}
{\includegraphics[scale=0.6]{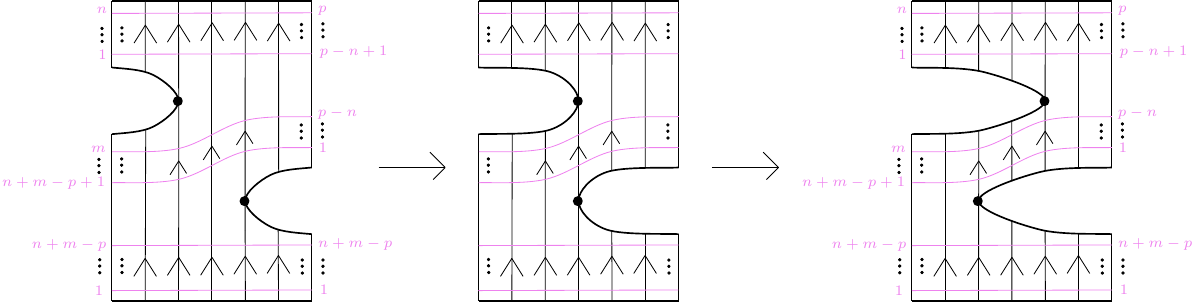}}
{The move $s_1$ for flow-spines enhanced to a move $s_1^{(i)}$ for graphs representing properly immersed curves.\label{si-1-move-proof:fig}}
\end{figure}
\begin{figure}
\faifig{}
{\includegraphics[scale=0.6]{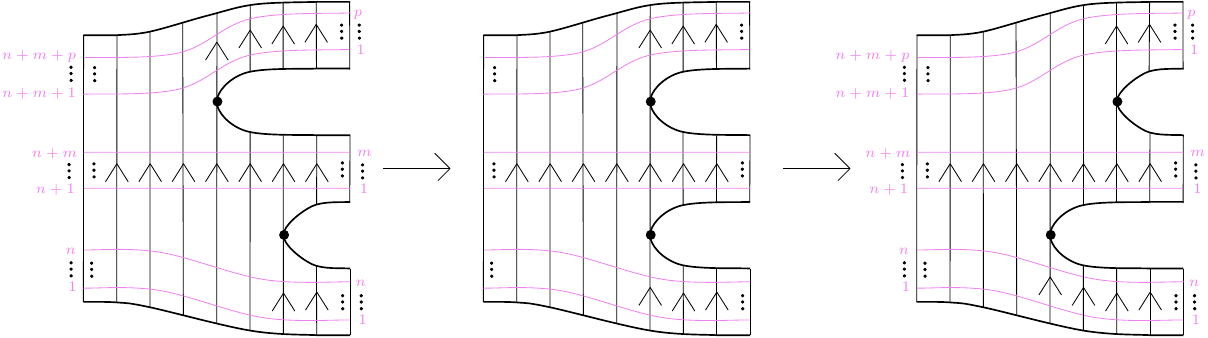}}
{The move $s_2$ enhanced to a move $s_2^{(i)}$.\label{si-2-move-proof:fig}}
\end{figure}
\begin{figure}
\faifig{}
{\includegraphics[scale=0.6]{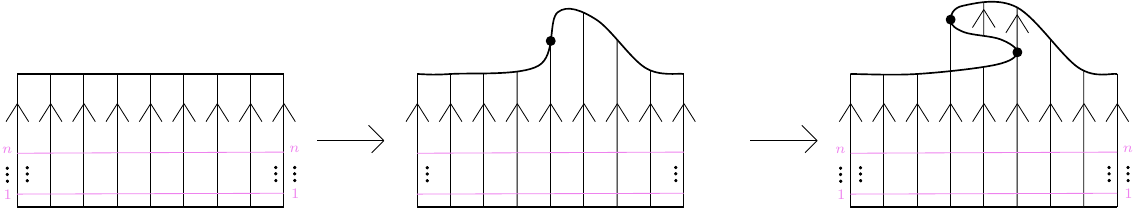}}
{The move  $s_3$ enhanced to a move $s_3^{(i)}$.\label{si-3-move-proof:fig}}
\end{figure}
The proof is complete.
\end{proof}

\begin{rem}\label{embedded:closed:pairs:rem}
\emph{In the spirit of Remarks~\ref{graphs:for:embedded:closed:curves:rem} and~\ref{moves:for:embedded:closed:curves:rem},
one can give a combinatorial presentation
of the set of pairs $(F,C)$ up to simultaneous homotopy, with $C$ either embedded, or closed, or closed and embedded, by
taking graphs in $\calG(\Genpairs)$ without vertices of type (3), or (4), or (3) and (4), and ignoring
the moves involving them.}
\end{rem}



\vspace{.5cm}

\noindent
Dipartimento di Matematica\\
Universit\`a di Pisa\\
Largo Bruno Pontecorvo, 5\\
56127 PISA -- Italy\\
\texttt{petronio@dm.unipi.it}

\end{document}